\documentclass{siamltex}

\usepackage{longtable}
\usepackage[utf8]{inputenc}
\usepackage[english]{babel}
\usepackage[linesnumbered, boxruled]{algorithm2e}
\usepackage{amsmath}
\usepackage{braket}  
\usepackage{amsfonts}
\usepackage{url}
\usepackage{xcolor} 
\usepackage{soul} 
\usepackage{booktabs}   

\usepackage{amssymb}
\usepackage{psfrag}
\usepackage{graphicx}
\usepackage{color}
\usepackage{caption}   
\usepackage{booktabs}   

\def\c{C}
\def\ck{C_k}
\newtheorem{ipotesi}{Assumption}[section]     
\def\argmin{\mathop{\rm argmin }}        

\usepackage{amsmath,bm,geometry,graphics}
\usepackage{color}

\newcommand{\eqdef}{\stackrel{\rm def}{=}}

\newcommand{\beqn}[1]{\begin{equation}\label{#1}}
\newcommand{\eeqn}{\end{equation}}
\renewcommand{\theequation}{\arabic{section}.\arabic{equation}}
\newcommand{\req}[1]{(\ref{#1})}

\newcounter{algo}[section]
\renewcommand{\thealgo}{\thesection.\arabic{algo}}
\newcommand{\algo}[3]{\refstepcounter{algo}
\begin{center}
\framebox[\textwidth]{
\parbox{0.95\textwidth} {\vspace{\topsep}
{\bf Algorithm \thealgo : #2}\label{#1}\\
\vspace*{-\topsep} \mbox{ }\\
{#3} \vspace{\topsep} }}
\end{center}}

\renewcommand{\Re}{\hbox{I\hskip -2pt R}}

\newcommand{\comment}[1]{}

\newcommand{\R}{\mathbb{R}}

\newcommand{\metodop}{ARC\_{\sc{SO}}}
\newcommand{\metodo}{ARC\_{\sc{SO}} }

\begin{document}
\title{Adaptive cubic regularization methods with dynamic 
inexact Hessian information and applications to finite-sum minimization}\thanks{Work partially supported by INdAM-GNCS under Progetti di Ricerca 2018.}

\author{Stefania Bellavia\footnotemark[2], Gianmarco Gurioli\footnotemark[3] and
Benedetta Morini\footnotemark[2]}

\renewcommand{\thefootnote}{\fnsymbol{footnote}}
\footnotetext[2]{Dipartimento  di Ingegneria Industriale, Universit\`a di Firenze,
viale G.B. Morgagni 40,  50134 Firenze,  Italia,
stefania.bellavia@unifi.it, benedetta.morini@unifi.it. Members of the INdAM Research Group GNCS.}

\renewcommand{\thefootnote}{\fnsymbol{footnote}}
\footnotetext[3]{Dipartimento  di Matematica e Informatica ``Ulisse Dini'', Universit\`a di Firenze,
viale G.B. Morgagni 67a,  50134 Firenze,  Italia,
gianmarco.gurioli@unifi.it. Member of the INdAM Research Group GNCS.}

\maketitle{}

\begin{abstract}
We consider the Adaptive Regularization with Cubics approach  for solving nonconvex optimization problems  
and propose a new variant based on inexact Hessian information chosen dynamically. 
The theoretical analysis of the proposed procedure is given. 
The key property of  ARC framework, constituted by optimal worst-case 
function/derivative  evaluation  bounds for first- and second-order critical point, is guaranteed.
Application to large-scale finite-sum minimization based on  subsampled Hessian is discussed and analyzed in both a deterministic  and probabilistic manner
and equipped with numerical experiments on synthetic and real datasets.
\end{abstract}

\section{Introduction}
Numerical methods based on the Adaptive Regularization with Cubics (ARC)  constitute an important 
class of Newton-type procedures for the solution of the
unconstrained, possibly nonconvex, optimization  problem
\begin{equation} 
\label{Pb}
\min_{x\in\mathbb{R}^n} f(x),
\end{equation}
\noindent
where $f:\mathbb{R}^n\rightarrow \mathbb{R}$ is smooth and bounded below.  Successively to the seminal works 
\cite{ARC1, ARC2}, ARC methods have become a very active area of research
due to their worst-case iteration and computational complexity bounds for achieving a desired level of accuracy 
in  first-order and second-order optimality conditions. Under reasonable assumptions on  $f$  and a
suitable realization of the adaptive cubic regularization method with derivatives of $f$ up to order 2,
Cartis et al. proved that an $(\epsilon,\epsilon_H)$ approximate first- and second-order critical point is found in at most 
$O(\max(\epsilon^{-3/2}, \epsilon_H^{-3}))$ iterations,
where $\epsilon$ and $\epsilon_H$ are positive  prefixed  first-order and second-order optimality tolerances, respectively
\cite{Toint1, ARC2, CGToint,CGTIMA};
this complexity result  is known to be sharp and optimal with respect to steepest descent, Newton's method and
Newton's method embedded into a linesearch or a trust-region strategy \cite{CGT, CGToint}.

Of particular practical interest is the ARC algorithm where exact second-derivatives of $f$ 
 are not required \cite{ARC1}.
Inexact Hessian information is used and suitable approximations of the Hessian make 
the algorithm  convenient for problems where the  evaluation of second-derivatives is expensive. Clearly,  the  
agreement between the Hessian and its approximation characterizes complexity and convergence rate behaviour of the procedure; 
in \cite{ARC1, ARC2}  the well-known Dennis-Mor\'e condition \cite{dm}  and slightly stronger agreements are considered.

Recently, Newton-type methods with  inexact Hessian information, and possibly inexact function and gradient information,
have received large attention see e.g., \cite{bmgt, bkkj,bkm, bbn, review,  byrd, CS, Cin, kl, pilanci, Roosta_in, Roosta,   Roosta_2p, Roosta_inexact}.  
The interest in  such methods is motivated
by problems where the derivative information about $f$ is computationally expensive,  such as 
large-scale optimization problems arising in machine learning and data analysis modeled as
\begin{equation}\label{finite_sum}
\min_{x\in\mathbb{R}^n} f(x)=\frac{1}{N}\sum_{i=1}^N{\phi_i(x)},
\end{equation}
with $N$ being a positive scalar and $\phi_i:\mathbb{R}^n\rightarrow \mathbb{R}$. 
Experimental studies have shown that second-order methods can be more efficient on badly-scaled or ill-conditioned problems 
than first-order methods even though inexact Hessian information is built via random  sampling methods, see e.g., \cite{bbn,
 review, Cin, kl,  Roosta_2p, Roosta_inexact}. In addition, these methods can take advantage of 
second-order  information to escape from saddle points \cite{CGToint,Roosta_2p}. 
ARC methods with  probabilistic  models have been proposed and studied in 
\cite{CS, Cin, kl,  Roosta, Roosta_2p, Roosta_inexact}, while a  cubic regularized method incorporating variance reduction techniques has been given in \cite{zhou_xu_gu}; much effort has been devoted to  weaken  the request on the level
of resemblance between the Hessian and its approximation though preserving optimal complexity bounds.

This work focuses on a variant of the ARC methods for problem (\ref{Pb})
with inexact Hessian information and presents a strategy
for choosing the Hessian approximation dynamically.  
We propose a rule for fixing the desired accuracy in the Hessian approximation and incorporate it 
into the ARC framework; the agreement between  the Hessian of $f$ and its approximation can be loose at the beginning of the iterative
process   and  
increases progressively as the norm of stepsize  drops below one and a stationary point for (\ref{Pb}) is approached. 
The  resulting ARC variant employs  a potentially milder accuracy requirement on 
the Hessian approximation  than  the proposals in  \cite{Cin, Roosta}, without impairing optimal complexity results. 
The new algorithm is analyzed theoretically and first- and second-order optimal complexity bounds are proved
in a deterministic manner; in particular, we show that  the complexity bounds  and convergence properties of our scheme match those of the ARC 
methods mentioned above. 
 Our proposal  has been motivated by the pervasiveness of  finite-sum minimization
problems  \eqref{finite_sum}   and the significant interest in unconstrained optimization methods with inexact Hessian information. 
Therefore, we discuss the application of our method to this relevant class of problems and show that it is  compatible with  
subsampled Hessian approximations adopted in literature; in this context, we give probabilistic and deterministic results as well as numerical results on a set of nonconvex binary classification problems.

The paper is organized as follows. In Section \ref{ARCbase} we briefly review the ARC framework, 
then in Section \ref{newalgo} we introduce our variant based on a dynamic rule for building the inexact Hessian.
The first-order iteration complexity  bound of the resulting algorithm is studied in Section \ref{complexity} along with 
the asymptotic behaviour of the generated sequence; complexity  bounds and convergence to second-order points are
analyzed in Section \ref{2ndo}. The application of our algorithm to the finite-sum optimization problem is discussed 
in Section \ref{fs}, while the relevant differences of our proposal from  the closely related works in the literature 
are discussed in Section \ref{cfr}. Finally, in Section \ref{numerical} we provide numerical results showing the effectiveness of our adaptive rule.

\vskip 5pt 
{\bf Notations.} The Euclidean vector and matrix norm is denoted as $\|\cdot \|$.
Given the scalar or vector or matrix $v$, and the non-negative scalar $\chi$, 
we write $v=O(\chi)$ if there is a constant $g$ such that $\|v\| \le  g \chi$.
Given any set ${\cal{S}}$, $|{\cal{S}}|$ denotes its cardinality.

\section{The adaptive regularization framework}\label{ARCbase} The ARC approach
for   unconstrained optimization, firstly proposed in \cite{g,np,w},  
is based on the use of a cubic model for $f$
and is a globally convergent second-order procedure.
If $f$ is smooth and the Hessian matrix $\nabla^2 f$ is globally Lipschitz continuous on $\mathbb{R}^n$ with 
$\ell_2$-norm Lipschitz constant $L$, i.e.,
$$
\|\nabla^2 f(x)-\nabla^2 f(y)\| \le L\| x-y\|,\quad \forall x,y\in\mathbb{R}^n,\quad \exists L>0,
$$
then  the Taylor's expansion of $f$  at $x_k\in\mathbb{R}^n$ 
with increment $s\in\mathbb{R}^n$ implies
\begin{equation}\label{mc}
\begin{split}
f(x_k+s)&\le
  f(x_k)+\nabla f(x_k)^T s +\frac{1}{2} s^T \nabla^2 f(x_k) s + \frac{L}{6} \| s \|^3 \eqdef m^C(x_k,s).
\end{split}
\end{equation}
Consequently, any step $s$ satisfying $ m^C(x_k,s)<  m^C(x_k,0)=f(x_k)$ provides a reduction of $f$ at $x_k+s$ 
with respect to the current value $f(x_k)$.

The ARC approach has received growing interest  starting from the papers by Cartis et al. \cite{ARC1,ARC2}  where  
it is not required the knowledge of either  exact second-derivatives of $f$ or the  Lipschitz constant $L$.
Specifically, the cubic model used at iteration $k$ has the form 
\begin{equation}\label{m}
m(x_k,s,\sigma_k)= f(x_k)+ \nabla f(x_k)^T s+\frac{1}{2} s^T B_k s+\frac{\sigma_k}{3}\|s\|^3\eqdef T_2(x_k,s)+\frac{\sigma_k}{3}\|s\|^3,
\end{equation}
where $B_k\in\mathbb{R}^{n\times n}$ is a symmetric approximation of $\nabla^2 f(x_k)$ and  $\sigma_k>0$ is 
the cubic regularization parameter  chosen adaptively to ensure the overestimation property as in (\ref{mc}).
The relevance of such procedure lies on its worst-case evaluation complexity for finding an $\epsilon$-approximate first-order critical point, 
i.e., a point $\widehat x$ such that
\begin{equation}\label{grad_eps}
\|\nabla f(\widehat x )\|\le \epsilon.
\end{equation}
In fact, in \cite{ARC2} worst-case iteration complexity of order  $O(\epsilon^{-3/2})$ is proved, 
provided that: (a) the step $s_k$ is the global minimizer of $m(x_k,s,\sigma_k)$
over a subspace of $\mathbb{R}^n$  including $\nabla f(x_k)$, see e.g. 
\cite{ blms, cd,ARC1}; (b)  the actual  objective decrease $f(x_k)-f(x_k+s_k)$   is a
prefixed fraction of the predicted model reduction $f(x_k)-m(x_k,s_k,\sigma_k)$, i.e., 
\begin{equation}\label{pi}
\pi_k=\frac{f(x_k)-f(x_k+s_k)}{f(x_k)-m(x_k,s_k,\sigma_k)}\ge \eta_1 ,
\end{equation}
for some $\eta_1\in (0,1)$; (c) the agreement between $\nabla^2 f(x_k)$ and  $B_k$ along $s_k$ is such that 
\begin{equation}\label{AM4}
\|(\nabla^2 f(x_k)-B_k)s_k\|\le \chi\|s_k\|^2,
\end{equation}
for all $k\ge 0$ and some constant $\chi>0$.


The requirement  \req{AM4} is stronger than the Dennis-Mor\'e condition \cite{dm} and it is unknown whether  it can be ensured theoretically  \cite{ARC2}. Kohler and Lucchi 
\cite{kl} suggested to achieve \req{AM4} 
by imposing
\begin{equation}\label{AMLK}
\| \nabla^2 f(x_k)-B_k\|\le \chi\|s_k\|.
\end{equation}
It is evident that the agreement between $\nabla^2 f(x_k)$ and $B_k$ depends on the steplength
which can be determined only after $B_k$ is formed. This issue is circumvented in practice employing the
steplength at the previous iteration \cite[\S 5]{kl}.

Xu et al. \cite{Roosta,Roosta_inexact}  analyzed ARC algorithm making  a major modification on the level of resemblance 
between $\nabla^2 f(x_k)$ and $B_k$ over (\ref{AM4}) requiring
\begin{equation}\label{AMXRM}
\|(\nabla^2 f(x_k)-B_k)s_k\|\le \mu \|s_k\|,
\end{equation}
with $\mu \in (0,1)$. 
In practice,  \req{AMXRM} is achieved imposing  $\| \nabla^2 f(x_k)-B_k \|\le \mu $.
In order to retain the optimal complexity of the classical ARC method, $\mu=O(\epsilon)$ is assumed.

We observe that, given  a positive $\upsilon$,   the requirement  $\| \nabla^2 f(x_k)-B_k \|\le \upsilon $ can be enforced 
approximating  $\nabla^2 f(x)$ by finite differences or  interpolating functions \cite{trbook}.  Moreover, for the 
class of large-scale finite-sum minimization (\ref{finite_sum}), the requirement  $\| \nabla^2 f(x_k)-B_k \|\le \upsilon $ can be 
satisfied   in probability via subsampling in Hessian computation, see e.g., \cite{bmgt, review, kl, Roosta}. 

A main advancement in ARC algorithm was obtained by Birgin et al. in the paper  \cite{Toint1} where 
ARC is generalized to higher order regularized models and  significant  modifications in the step computation and
acceptance criterion {are introduced with  respect to \cite{ARC1,ARC2}.
The Algorithm \ref{ARCbraz} detailed below is proposed in \cite{Toint1} and here restricted to the version based on second order model and cubic regularization;
as in \cite{Toint1} $B_k$ is supposed to be equal to $ \nabla^2 f(x_k)$. 
Remarkably, global optimization of $m(x_k,s,\sigma_k)$ over  a subspace of $\mathbb{R}^n$
is no longer required and  conditions (\ref{stepbraz1})--(\ref{stepbraz2}) on  the step $s_k$  are quite standard in unconstrained optimization 
when  a model is approximately minimized. A further distinguishing feature is that
the denominator in (\ref{rho}) involves the  second-order Taylor expansion 
of $f$ without the regularizing term, whereas the  denominator in (\ref{pi})  involves the cubic model $m(x_k,s,\sigma_k)$ itself.
Analogously to the algorithm in \cite{ARC2}, Algorithm  \ref{ARCbraz} finds an 
$\epsilon$-approximation first-order critical point in at most $O(\epsilon^{-3/2})$ 
evaluations of $f$ and its derivatives $\nabla f$, $\nabla^2 f$  (\cite{Toint1}).
\vskip 3pt
\algo{ARCbraz}{ARC algorithm  \cite{Toint1}}{
\textbf{Step \boldmath$0$: Initialization}. Given an initial point $x_0$, the initial regularizer $\sigma_0>0$,  the 
accuracy level $\epsilon$. Given  $\theta$, $\eta_1$, $\eta_2$, $\gamma_1$, $\gamma_2$, $\gamma_3$, $\sigma_{\textrm{min}}$ s.t.
\[
\theta>0, \quad \sigma_{\min}\in (0, \sigma_0], \quad  0<\eta_1\le \eta_2<1, \quad 0<\gamma_1<1<\gamma_2<\gamma_3.
\]
Compute $f(x_0)$ and set $k=0$.
\vskip 2pt
\textbf{Step \boldmath$1$: Test for termination}. Evaluate $\nabla f(x_k)$. If $\|\nabla f(x_k)\|\le \epsilon$, terminate with the
approximate  solution $\widehat x= x_k$. Otherwise, compute $B_k=\nabla^2 f(x_k)$.
\vskip 2pt
\textbf{Step \boldmath$2$: Step computation.} 
Compute the step $s_k$ by approximately minimizing the model $m(x_k,s,\sigma_k)$ w.r.t. $s$ so that
\begin{eqnarray}
& & m(x_k,s_k,\sigma_k)<m(x_k,0,\sigma_k),\label{stepbraz1}\\
& & \|\nabla_s m(x_k,s_k,\sigma_k)\|\le \theta \|s_k\|^2.\label{stepbraz2}
\end{eqnarray}
\vskip 2pt
\textbf{Step \boldmath$3$: Acceptance of the trial step.}  Compute $f(x_k+s_k)$ and define
\begin{equation}
\label{rho}
\rho_k=\frac{f(x_k)-f(x_k+s_k)}{T_2(x_k,0)-T_2(x_k,s_k)}.
\end{equation}
If $\rho_k\ge \eta_1$, then define $x_{k+1}=x_k+s_k$; otherwise define $x_{k+1}=x_k$.
\vskip 2pt
\textbf{Step \boldmath$4$: Regularization  parameters update.}  Set 
\begin{equation}
\label{sigma}
\sigma_{k+1}\in\left \{ \begin{array}{ll}
\left[\mbox{max}(\sigma_{\textrm{min}},\gamma_1\sigma_k),\sigma_k\right], & \mbox {if } \rho_k\ge \eta_2,\\
\left[\sigma_k,\gamma_2\sigma_k \right], & \mbox{if } \rho_k\in\left[ \eta_1,\eta_2\right),\\
\left[\gamma_2\sigma_k,\gamma_3\sigma_k\right],& \mbox{if } \rho_k< \eta_1.
\end{array}
\right .
\end{equation}
Set $k=k+1$ and go to Step $1$ if $\rho_k\ge \eta_1$, or to Step 2 otherwise.
\vskip 2pt
}
\vskip 15pt

In this work, we propose a variant of  Algorithm \ref{ARCbraz} employing a model of the form (\ref{m}) and a 
matrix $B_k$ such that
\begin{equation}\label{diffck}
\|\nabla^2 f(x_k)-B_k\|\le C_k ,
\end{equation}
for all $k\ge 0$ and   positive scalars $C_k$. 
The accuracy $C_k$  on   the inexact Hessian information is dynamically chosen and 
when the norm of the step is smaller than one  it depends on the current gradient's norm.  
We will show that for properly chosen scalars $C_k$, condition   \eqref{diffck} is an implementable rule to achieve (\ref{AM4}). 
In addition, in the first phase of the procedure 
the accuracy imposed on $B_k$ can be less stringent with  respect to the proposal  made in \cite{Cin, Roosta, Roosta_inexact},
though  preserving the complexity bound $O(\epsilon^{-3/2})$. 
In the subsequent  sections   we present and study our variant of the ARC algorithm. We refer  to  Sections \ref{fs}  and
\ref{cfr} for a discussion  on the application to the finite-sum optimization problem and the comparison with the above 
mentioned related works  in the literature.

\section{An adaptive choice of the inexact  Hessian}\label{newalgo}
In this section, we propose and study a variant of  Algorithm \ref{ARCbraz} which maintains  the complexity bound  $O(\epsilon^{-3/2})$. 
Our algorithm is  based on the use of an approximation
$B_k$ of $\nabla^2 f(x_k)$ in the construction of the cubic model and a rule for choosing the level of agreement 
between $B_k$ and  $\nabla^2 f(x_k)$. The accuracy requirements
in the approximate minimization of $m(x_k, s, \sigma_k)$ consist of (\ref{stepbraz1})  and  a condition on
$\|\nabla_s m(x_k,s_k,\sigma_k)\|$ which  includes condition  (\ref{stepbraz2})  but it is not limited to it.

{Our analysis is carried out under the following Assumptions on the function $f$ and the matrix $B_k$ used in the model  (\ref{m}).
\vskip 5pt
\begin{ipotesi}\label{ass_f}
The objective function $f$ is twice continuously differentiable on $\mathbb{R}^n$ and its  Hessian
is Lipschitz continuous on the path of iterates with Lipschitz constant $L$,
$$
\|\nabla^2 f(x_k+\beta s_k)-\nabla^2 f(x_k)\| \le L\beta \| s_k\|,\quad \forall k\ge 0,\quad  \beta\in [0,1].
$$
\end{ipotesi}
}
\vskip 5pt
\begin{ipotesi}\label{ass_Bk}
For all $k\ge 0$ and    some $\kappa_B  \ge 0$, it holds 
$$
\|B_k\|\le \kappa_B.
$$
\end{ipotesi}

Further, we suppose that the step $s_k$ computed  has the following properties.  
\vskip 5pt
\begin{ipotesi}\label{ass_m}
For all $k\ge 0$ and    some $0\le \theta_k\le \theta$, $\theta \in [0,1)$, $s_k$ satisfies
\begin{eqnarray}
& & m(x_k,s_k,\sigma_k)<m(x_k,0,\sigma_k),\label{mdesc} \\
& & \| \nabla_s m(x_k,s_k,\sigma_k)\| \le \theta_k \|\nabla f(x_k)\|. \label{tc}  
\end{eqnarray}
\end{ipotesi}
\vskip 5pt
\noindent
By (\ref{mc}) and (\ref{m}) it easily follows
\begin{equation}\label{mc_T2}
m^C(x_k,s)=T_2(x_k,s)+\frac{1}{2} s^T (\nabla^2 f(x_k)-B_k) s + \frac{L}{6} \| s \|^3.
\end{equation}
Then,   (\ref{mc}) yields
\begin{equation}\label{fsppbound2}
f(x_k+s)\le     T_2(x_k,s)  +E_k(s) ,   
\end{equation}
where 
\begin{equation}\label{e}
E_k(s)=\frac 1 2  \| \nabla^2 f(x_k)-B_k \|  \| s\|^2+ \frac{L}{6} \| s \|^3.
\end{equation}
Now, we make our key requirement on the agreement between  $B_k$ and  $\nabla^2 f(x_k)$  and analyze its effects on ARC algorithm.
\vskip 5pt
\begin{ipotesi}\label{ass_Dk}
Let $B_k\in \mathbb{R}^{n\times n}$  satisfy 
\begin{eqnarray}
\Delta_k&=&\nabla^2 f(x_k)-B_k, \quad  \|\Delta_k\| \le  \ck,  \label{Dk}\\
  \ck&=& \c,  \hspace*{82pt}  \mbox{if } \ \ \|s_k\|\ge 1,  \hspace*{86pt}  \label{bound1}\\
  \ck &\le& \alpha(1-\theta)\| \nabla f(x_k)\|, \quad    \mbox{if } \ \ \|s_k\|< 1 ,  \label{Ckbound} 
\end{eqnarray}
for all $k \ge 0$, with  $\alpha, \, C_k$ and $C$  positive scalars, $s_k\in \mathbb{R}^n$ and  $\theta\in [0,1)$ 
 as  in Assumption \ref{ass_m}.
\end{ipotesi}
\vskip 5pt

Bounds on  $\|\Delta_k\|$ and on $E_k(s_k)$ involving $\|s_k\|$ are derived below and give $E_k(s_k)=O(\|s_k\|^3)$.
\vskip 5pt
\begin{lemma}\label{DE}
Let    Assumptions \ref{ass_f}--\ref{ass_Dk} 
hold, and let $E_k(s)$  and  $\Delta_k$  as in (\ref{e}), (\ref{Dk}). Then
\begin{equation}\label{Dkcases} 
\| \Delta_k \| \le 
\left \{ \begin{array}{ll}
C\|s_k \|,  & \mbox{ if } \, \|s_k\|\ge 1, \\
\alpha (\kappa_B+\sigma_k) \|s_k\|, & \mbox{ if }\,  \|s_k \|< 1,
\end{array} 
\right.
\end{equation}
and   
\begin{equation}\label{Ekcases}
  E_k (s_k) \le 
\left \{ \begin{array}{ll}
\displaystyle 
\frac{1}{2} \left(C+\frac{L}{3}\right) \|s_k \|^3, & \mbox{ if } \, \|s_k\|\ge 1,  \\
\displaystyle 
\frac 1 2 \left(\alpha ( \kappa_B+\sigma_k) +\frac{L}{3} \right) \| s_k \| ^3, & \mbox{ if } \, \|s_k \|< 1 .
\end{array} 
\right.
\end{equation}
\end{lemma}
\begin{proof}
\noindent
First consider  the case $\| s_k \|\ge 1$. Trivially, the inequality in   (\ref{Dk})  gives  (\ref{Dkcases}) and 
$$
E_k(s_k)\le \frac{1}{2}C \|s_k \|^3+\frac{L}{6}\|s_k\|^3,
$$
i.e., the first bound in  (\ref{Ekcases}).

Suppose now that $\|s_k\|<1$. 
Using (\ref{tc}),  Assumptions   \ref{ass_Bk} and \ref{ass_m}, we obtain 
\begin{eqnarray}
 \theta \|\nabla f(x_k)\|&\ge &\|\nabla_s m(x_k, s_k, \sigma_k)\| \nonumber \\
& =& \left\| \, \nabla f(x_k) + B_k s_k +\sigma_k s_k \|s_k\| \,  \right\|  \label{derivm}\\
 & \ge&  \|\nabla f(x_k) \| - \|B_k\|\, \| s_k \| -\sigma_k  \| s_k\|^2 \nonumber\\
 &\ge & \|\nabla f(x_k) \|-\kappa_B \|s_k \| -\sigma_k \|s_k\|, \nonumber
\end{eqnarray}
which gives 
\begin{equation} \label{supperbound}
\| s_k\| \ge \frac{(1-\theta)\|\nabla f(x_k)\|}{\kappa_B+\sigma_k}.
\end{equation}
Thus,   (\ref{Dk}) and \eqref{Ckbound} yield
$$
\|\Delta_k\|  \le  \ck  = \frac{\ck}{\| s_k \|}\|s_k  \| \le  \frac{\ck(\kappa_B+\sigma_k)}{(1-\theta)\| \nabla f(x_k)  \|}\| s_k  \|.
$$
Finally,  (\ref{Ckbound}) implies (\ref{Dkcases}) and this along with  (\ref{e}) gives  (\ref{Ekcases}).
\end{proof}
\vskip 5pt
Taking into account the previous result and assuming $\alpha \in [0,2/3)$, we can establish when the overestimation property
$f(x_k+s_k)\le m(x_k, s_k, \sigma_k)$ is verified. Using (\ref{m}),  (\ref{mc_T2}),  (\ref{fsppbound2}) we see that if 
$E_k(s_k)\le \sigma_k\|s_k\|^3/3$,  then $m^C(x_k,s_k)\le m(x_k, s_k, \sigma_k)$ which implies that $ m(x_k, s_k, \sigma_k)$   overestimates $f(x_k+s)$. 

If $\|s_k\| \ge 1$ and   
$$
\frac{1}{2}\left(C+\frac{L}{3}\right)\le \frac{\sigma_k}{3},  \quad  \mbox { i.e., }\quad \sigma_k\ge\frac{3C+L}{2},
$$
then (\ref{Ekcases}) implies   $m^C(x_k,s_k)\le m(x_k, s_k, \sigma_k)$.
Analogously, if $\|s_k\|<1$
$$
\frac 1 2 \left(\alpha ( \kappa_B+\sigma_k) +\frac{L}{3} \right) \le \frac{\sigma_k}{3}   \quad \mbox{ i.e., }\quad  \sigma_k\ge \frac{3\alpha \kappa_B +L}{2-3\alpha},
$$
then (\ref{Ekcases}) implies   $m^C(x_k,s_k)\le m(x_k, s_k, \sigma_k)$.

We can  now  deduce an important upper bound on the regularization parameter $\sigma_k$.
\vskip 5pt
\begin{lemma}\label{Lsigmamax}
Let   Assumptions \ref{ass_f}--\ref{ass_Dk} hold. Suppose that 
the scalar $\alpha$ in Assumption \ref{ass_Dk} is such that 
$\alpha \in \left[0, \displaystyle \frac 2 3\right)$ 
and that the constant $\eta_2$ in Algorithm \ref{ARCbraz} is such that
$\eta_2\in \displaystyle \left(0, \frac{2-3\alpha}{2}\right )$. Then it holds
\begin{equation}\label{sigmamax}
\sigma_k \le \sigma_{\max}\eqdef \max \left\{\sigma_0,\, \gamma_3 \frac{3C+L}{2(1-\eta_2)},\,
\gamma_3 \frac{3\alpha\kappa_B+L}{2-3\alpha-2\eta_2}\right\}  \quad \forall k\ge 0,
\end{equation}
where $\gamma_3$ is the constant used in \eqref{sigma}. 
\end{lemma}
\begin{proof}
 Let us derive  conditions on $\sigma_k$  ensuring    $\rho_k\ge \eta_2$.
By (\ref{m}) and (\ref{mdesc})  it  follows  $\|s_k\|\neq0$ and 
\begin{equation}\label{m_T2}
0< m(x_k,0,\sigma_k)-m(x_k,s_k,\sigma_k)=T_2(x_k,0 )-T_2(x_k,s_k)- \frac{\sigma_k}{3}\| s_k \|^3.
\end{equation}
Thus
\begin{equation}\label{diffT}
T_2(x_k,0 )-T_2(x_k,s_k)> \frac{\sigma_k}{3}\|s_k \|^3>0,
\end{equation}
and by (\ref{fsppbound2}) and the fact that $E_k(s_k)> 0$
\begin{eqnarray}
1-\rho_k =   \frac{   f(x_k+s_k)-T_2(x_k,s_k) }{ T_2(x_k,0)-T_2(x_k,s_k) }  
\le \frac{E_k(s_k)}{ T_2(x_k,0)-T_2(x_k,s_k) } <\frac{3 E_k(s_k)}{\sigma_k\| s_k \|^3}. \label{succ1} 
\end{eqnarray}
If $\| s_k \|\ge 1$, using  (\ref{Ekcases})   we obtain
\[
1-\rho_k  {<}  \frac{3}{2\sigma_k} \left(C+\frac{L}{3}\right),
\]
and $\rho_k\ge \eta_2$ is guaranteed when 
\[
\sigma_k \ge \frac{3C+L}{2(1-\eta_2)} .
\]
On the other hand, if $\| s_k \| < 1$ then (\ref{Ekcases}) and \eqref{succ1} give
\[
1-\rho_k   {<}  \frac{3}{2\sigma_k}  \left(\alpha ( \kappa_B+\sigma_k) +\frac{L}{3} \right),
\]
and  $\rho_k\ge \eta_2$ is guaranteed when 
\[
\sigma_k\ge \frac{3\alpha\kappa_B+L}{2-3\alpha-2\eta_2} ,
\]
noting that the denominator is strictly positive by assumption.
Then, the updating rule (\ref{sigma}) implies  $\sigma_{k+1}\le \sigma_k$ in case $\rho_k\ge \eta_2$  and,
more generally,   inequality  (\ref{sigmamax}).
\end{proof}
\vskip 5pt
An important consequence of  Lemma \ref{DE} and  Lemma \ref{Lsigmamax} is that  \eqref{diffck} implies 
\begin{equation}\label{bound_cost}
\|\nabla^2 f(x_k)-B_k\|\le \max(C, \, \alpha(\kappa_B+\sigma_{\max}))\|s_k\|,
\end{equation}
for all $k\ge 0$,   and consequently  condition  (\ref{AM4}) is satisfied.

In Lemma \ref{Lsigmamax}, the value of $\alpha$ in (\ref{Ckbound}) determines the accuracy of $B_k$ as an approximation to $\nabla^2 f(x_k)$
and the admitted maximum value of $\eta_2$. For decreasing values of $\alpha$, the accuracy of the Hessian approximation increases
and $\eta_2$ reaches one. On the other hand, if $\alpha$ tends to $\displaystyle \frac 2 3$ then the accuracy of the Hessian approximation 
reduces, $\eta_2$  tends to zero and $\sigma_{\max}$ tends to infinity.\footnote{Values  $\eta_2= \displaystyle \frac 3 4$ and 
$\eta_2= \displaystyle \frac 9 {10}$ used in the literature for 
 the trust-region and ARC frameworks are achieved setting  $\alpha= \displaystyle \frac 1 6 $  and
{$\alpha= \displaystyle \frac 1 {15}$},   respectively.} 

On the base of the previous analysis we sketch  our version of Algorithm \ref{ARCbraz} denoted as Algorithm \ref{ARCalgonew}.
The main feature is the adaptive rule for adjusting the agreement between  $B_k$ and $\nabla^2 f(x_k)$  as specified in
Assumption \ref{ass_Dk}.
 At the beginning of $k$th iteration, the variable ${\rm  flag}$ is equal to either 1 or 0 and  determines the value of $C_k$; specifically
$C_k=C$ if ${\rm flag}=1$, $C_k=\alpha(1-\theta)\|\nabla f(x_k)\|$  otherwise  with  $\nabla f(x_k)$ being available (at iteration $k=0$, flag is set equal to $1$).
Scalars $C$ and  $\alpha$ are initialized at Step 0; the choice of $\alpha$ and $\eta_2$  is made in accordance to the   results presented above. 
Then, $B_k$ is computed at Step 2 and the trial step $s_k$ is computed at Step 3.

Step 4 is devoted to a check on the accordance between $C_k$ and $\|s_k\|$   since (\ref{Ckbound})  is required to hold if $\|s_k\|<1$ 
whereas $\|s_k\|$ can be determined only after $B_k$ is formed.  Therefore, at the end of a successful iteration  the value of ${\rm flag}$ is fixed 
accordingly to the steplength of the last step. Successively, once $B_k$ and   $s_k$ have been computed,   if
$\| s_k \|< 1$,   ${\rm flag}=1$ and  $C>\alpha(1-\theta)\|\nabla f(x_k)\|$ hold,  then the step is rejected and the iteration is {\em unsuccessful};
variable ${\rm flag}$ is set  equal to $0$ and  $B_k$ is recomputed at the successive iteration. This unsuccessful iteration is ascribed to  
the choice of matrix $B_k$, hence the regularization parameter is left unchanged. 
On the other hand, if the level of accuracy in matrix $B_k$ with respect to $\nabla^2 f(x_k)$ fulfills the requests (\ref{bound1})--(\ref{Ckbound}),  
in Step 5 we proceed  for acceptance of the trial steps and  update of  the regularizing parameter  as in Algorithm \ref{ARCbraz}. 
Summarizing, by construction,  {\em  Assumption \ref{ass_Dk} is satisfied at every successful iteration and at any unsuccessful iteration detected in Step 5. }

\vskip 5pt
\algo{ARCalgonew}{ARC algorithm with dynamic   Hessian  accuracy}{
\textbf{Step \boldmath$0$: Initialization}. Given an initial point $x_0$, the initial regularizer $\sigma_0>0$,  the 
accuracy level $\epsilon$. Given   $\theta$, $\alpha$,  $\eta_1$, $\eta_2$, $\gamma_1$, $\gamma_2$, $\gamma_3$, $\sigma_{\textrm{min}}$, $C$ s.t.
\[
 0<\theta<1, \; \alpha\in  \left[0, \displaystyle \frac 2 3\right),  \; \sigma_{\min}\in (0, \sigma_0], \;  0<\eta_1\le \eta_2 < \frac{2-3\alpha}{2}, \; 0<\gamma_1<1<\gamma_2<\gamma_3,\;
C>0
\]
\\
Compute $f(x_0)$ and set  $k=0$, $C_0=C$, ${\rm flag}=1$.\\

\textbf{Step \boldmath$1$: Test for termination}. If $\|\nabla f(x_k)\|\le \epsilon$, terminate with the current solution $\widehat x= x_k$.  
\vskip 2pt
\textbf{Step \boldmath$2$: Hessian approximation.} Compute  $B_k$  satisfying \eqref{Dk}.
\vskip 2pt
\textbf{Step \boldmath$3$: Step computation.} Choose $\theta_k\le \theta$. Compute the step $s_k$ satisfying (\ref{mdesc}) and (\ref{tc}).
\vskip 2pt
\textbf{Step \boldmath$4$: Check on $\|s_k\|$. }\\
If $\| s_k \|< 1$ and ${\rm flag}=1$  and $C>\alpha(1-\theta)\|\nabla f(x_k)\|$ 
\begin{itemize}
\item [] set $x_{k+1}=x_k$, $\sigma_{k+1}=\sigma_k$,\quad (\textit{unsuccessful iteration}) 
\item [] set $C_{k+1}=\alpha(1-\theta)\|\nabla f(x_k)\|$, ${\rm flag}=0$,
\item[] set $k=k+1$ and go to Step $2$.
\end{itemize}
\vskip 2pt
\textbf{Step \boldmath$5$: Acceptance of the trial step and parameters update.}  \\
Compute $f(x_k+s_k)$ and  $\rho_k$ in (\ref{rho}).
If $\rho_k\ge \eta_1$ 
\begin{itemize}
\item[] define $x_{k+1}=x_k+s_k$, set
$$
\sigma_{k+1}\in\left \{ \begin{array}{ll}
\left[\mbox{max}(\sigma_{\textrm{min}},\gamma_1\sigma_k),\sigma_k\right], & \mbox {if } \rho_k\ge \eta_2,\quad~~~~~~ \textit{(very successful iteration)}\\
\left[\sigma_k,\gamma_2\sigma_k \right], & \mbox{if } \rho_k\in\left[ \eta_1,\eta_2\right),\quad \textit{(successful iteration)}\\
\end{array}
\right .
$$
 \item[] If $\|s_k \|\ge 1$ set $C_{k+1}=C$,  ${\rm flag}=1$.  \\
 \hspace*{5pt} Otherwise set $C_{k+1}=\alpha(1-\theta)\|\nabla f(x_{k+1})\|$, ${\rm flag}=0$.
 \item [] Set $k=k+1$ and go to Step $1$.

\end{itemize}
else
 \begin{itemize}
 \item[] define $x_{k+1}=x_k$, $\sigma_{k+1}\in \left[\gamma_2\sigma_k,\gamma_3\sigma_k\right], \quad $  (\textit{unsuccessful iteration})
 \item[] $C_{k+1}=C_k$, $B_{k+1}=B_k$,
 \item [] set $k=k+1$ and go to Step $3$.
 \end{itemize}
}
\vskip 10 pt

Finally,   both ${\rm flag}$ and $C_k$ are updated in Step 5 as follows.
If the iteration is {\em successful},  we update   ${\rm flag}$ and $C_k$ 
following (\ref{bound1})--(\ref{Ckbound})  and using the norm of  the accepted trial step; 
clearly,  this is a prediction  as the step $s_{k+1}$ is not available at 
this stage  and such a setting may be rejected 
at Step 4 of the successive iteration. If the iteration is {\em unsuccessful}, then we do not change either $C_k$ or $B_k$.

\vskip 10pt

The classification of successful and unsuccessful iterations of the Algorithm \ref{ARCalgonew} between 0 and $k$   can be made introducing the sets
\begin{eqnarray}
\mathcal{S}_k &=&\Set{0\le j\le k | j \mbox{ successful in the sense of Step 5}},\\
 \mathcal{U}_{k,1}&=&\Set{0\le j\le k | j \mbox{ unsuccessful in the sense of Step 5}},\label{uk1}\\
\mathcal{U}_{k,2}&=&\Set{0\le j\le k | j \mbox{ unsuccessful in the sense of Step 4}}\label{uk2}.
\end{eqnarray}

More insight into the settings of $C_k$ and $\sigma_k$ in our algorithm, first we note  that  $C_k$  satisfies
$$
C_k=\alpha\omega(s_k) (1-\theta) \|\nabla f(x_k) \| + (1-\omega(s_k))C,
$$
where $\omega: W \rightarrow \Set{0,1}$ denotes the characteristic function of  $W=\{s_k: \|s_k\| <1\}$. 
It follows that  if
$$
\|\nabla f(x)\|\le \kappa_g,
$$
for all $x$ in an open convex set $X$ containing   $\{x_k\}$ and some positive $\kappa_g$, then $C_k\le \max\{C, \, \alpha(1-\theta)\kappa_g\}$.

Second, we observe that the update of $\sigma_k$ is not affected by unsuccessful iterations in the sense of Step 4. In fact, we have
$\sigma_{k+1}=\sigma_k$ whenever an unsuccessful iteration occurs at Step 4 and  the rule  for adapting $\sigma_j$,  $j\le k$, has the form 
\begin{eqnarray}
 \sigma_{j+1}&\ge& \gamma_1\sigma_j ,    \ \quad  j\in {\cal{S}}_k\label{newsigma1},\\
 \sigma_{j+1}&\ge& \gamma_2 \sigma_j ,   \    \quad  j\in  {\cal{U}}_{k,1} \label{newsigma2},\\
  \sigma_{j+1} &=&\sigma_j,   \   \qquad  j\in  {\cal{U}}_{k,2}.\label{newsigma3}
\end{eqnarray}
As a consequence, the upper bound on the scalars $\sigma_k$ established in Lemma \ref{Lsigmamax} is still valid.

\section{Complexity analysis}\label{complexity}
In this section we study the iteration complexity of   Algorithm \ref{ARCalgonew} assuming that $f$ is bounded below, i.e., there exists
$f_{low}$ such that  
$$ 
f(x)\ge f_{low},~\forall x\in\mathbb{R}^n.
$$ 

We consider two possible stopping criteria for
the approximate minimization of model $m_k$ at Step 3. Given $\theta\in (0,1)$, the first criterion has the form  
\begin{equation}
\label{tcsub}
\| \nabla_s m(x_k,s_k,\sigma_k)\| \le \theta \min \left( \| s_k\|^2, \|\nabla f(x_k)\|  \right), 
\end{equation}
which amounts to (\ref{tc}) with $\theta_k= \theta \min \left( 1, \frac{\| s_k\|^2}{\|\nabla f(x_k)\|} \right)$. 
The  second criterion is  considered in \cite[Eqn. (3.28)]{ARC1} and takes the form
\begin{equation}
\label{tc.s}
 \| \nabla_s m(x_k,s_k,\sigma_k)\| \le \theta \min(1,\|s_k\|) \|\nabla f(x_k)\|. 
\end{equation}
It corresponds to the choice $\theta_k= \theta \min(1,\|s_k\|)$ in (\ref{tc}).
\vskip 5pt
\begin{lemma}
\label{sk}
 Let Assumptions \ref{ass_f} and \ref{ass_Bk} hold.
Suppose that  $\alpha \in \left[0, \displaystyle \frac 2 3\right)$ and 
$\eta_2\in \displaystyle \left(0, \frac{2-3\alpha}{2}\right )$ in Algorithm \ref{ARCalgonew}.
Then, at   iteration $k\in \mathcal{S}_k\cup \mathcal{U}_{k,1}$ 
$$
\|s_k \| \ge \sqrt{\zeta \|\nabla f(x_k+s_k)\|},
$$
for some  positive $\zeta$, both when $s_k$ satisfies (\ref{tcsub})   and when =
$s_k$ satisfies (\ref{tc.s}) and the norm of the Hessian is bounded above 
by a constant $\kappa_H$ on the path of iterates,
\begin{equation}\label{bound_H}
\|\nabla^2 f(x_k+\beta s_k)\| \le \kappa_H,\quad \forall k\ge 0,\quad  \beta\in [0,1].
\end{equation}
\end{lemma}
\begin{proof}
Taylor expansions of $f$ and $\nabla f$ give
\begin{eqnarray}
 f&(x_k+s)= f(x_k)+s^T \nabla f(x_k)+\frac{1}{2}s^T\nabla^2 f(x_k)s+\int_0^1 (1-\tau)s^T(\nabla^2 f(x_k+\tau s)-\nabla^2 f(x_k))s\,d\tau,\nonumber\\
&\nabla f(x_k+s_k)=\nabla f(x_k)+\int_0^1\nabla^2 f(x_k+t s_k )s_k dt. \label{diffgrad}
\end{eqnarray}
Then, noting that the assumptions of  Lemma \ref{Lsigmamax} hold  at iterations $k\in \mathcal{S}_k \cup \mathcal{U}_{k,1}$, using the Lipschitz continuity of $\nabla^2 f$, (\ref{bound1}),  (\ref{Dkcases}) 
(valid at $k\in \mathcal{S}_k\cup \mathcal{U}_{k,1}$) and  (\ref{sigmamax}),  we derive
\begin{eqnarray}
\|\nabla f(x_k+s_k)-\nabla_s T_2(x_k,s_k) \| &= &\| \nabla f(x_k+s_k)-\nabla f(x_k)-B_k s_k\| \nonumber\\
&\le& \|\Delta_k s_k\|+ \int_0^1 \| (\nabla^2 f(x_k+\tau s_k)-\nabla^2 f (x_k)) s_k\|\,d\tau \nonumber \\
&\le&  \|\Delta_k \| \|s_k\| + \frac{L}{2}   \| s_k\|^2 \nonumber \\
&\le& \left( \max\left(C,  \alpha( \kappa_B+\sigma_{\max}) \right)+\frac{L}{2}  \right)\| s_k \|^2.\label{dis2}
\end{eqnarray}

Moreover, by (\ref{derivm})
\begin{eqnarray}
\nabla f(x_k+s_k)&=&\nabla f(x_k+s_k)-\nabla_s T_2 (x_k,s_k)+\nabla_s T_2(x_k,s_k)+\sigma_k\| s_k\| s_k-\sigma_k\| s_k\| s_k \nonumber \\
&=&  \nabla f(x_k+s_k)-\nabla_s T_2(x_k,s_k)+ \nabla_s m(x_k, s_k, \sigma_k) -\sigma_k\| s_k\| s_k.\label{dis1}
\end{eqnarray}

Now consider the case  $s_k$ satisfying (\ref{tcsub}).  
Condition \eqref{tcsub} along with (\ref{dis1}) and (\ref{dis2})  yield 
\begin{eqnarray*}
\|\nabla f(x_k+s_k)\|
&\le& \left(\max \left( C, \alpha(\kappa_B+\sigma_{\max})  \right)+ \frac{L}{2}+\theta+\sigma_{\max}\right) \| s_k \|^2,
\end{eqnarray*} 
which  gives the claim with $\zeta=1/\left( \max\ \left(C, \alpha(\kappa_B+\sigma_{\max})  \right)+L/2+\theta+\sigma_{\max}  \right)$. 

We turn now  the attention to the case $s_k$ satisfying (\ref{tc.s}).   Combining  \eqref{diffgrad} 
and the boundness of the Hessian we have 
\[
\| \nabla f(x_k) \|\le \| \nabla f(x_k+s_k)\| + \kappa_H\|s_k\|,
\]
 and  by (\ref{tc.s})
\begin{eqnarray*}
\|\nabla_s m(x_k, s_k, \sigma_k)\| &\le & \theta \min(1,\|s_k\|) \|\nabla f(x_k+s_k)\| +\theta \min(1,\|s_k\|) \kappa_H\|s_k\|\\
& \le & \theta  \|\nabla f(x_k+s_k)\| +\theta  \kappa_H \|s_k\|^2.
\end{eqnarray*}
Thus,   (\ref{dis2})   and  (\ref{dis1})   give 
$$
(1-\theta)\|\nabla f(x_k+s_k)\|\le  \left( \max\left(C,  \alpha(\kappa_B+\sigma_{\max}) \right)+ L/2+\theta \kappa_H+\sigma_{\max}  \right)    \|s_k\|^2,
$$
and  the claim follows with $\zeta= (1-\theta)/(\max \left(C,\alpha(\kappa_B+\sigma_{\max})\right)+ L/2+\theta \kappa_H+\sigma_{\max}) $. 
\end{proof} 
\vskip 5pt

\begin{theorem}
\label{TComplexityExact}
Suppose  that $f$ in \eqref{Pb} is lower bounded by $f_{\mbox{\scriptsize low}}$
and the assumptions of Lemma \ref{sk} hold.
Then  Algorithm \ref{ARCalgonew} requires at most
\begin{equation}\label{succfinale}
{\cal{I}}_S=\left\lfloor \kappa_s\frac{f(x_0)-f_{\mbox{\scriptsize low}}}{\epsilon^{3/2}}\right\rfloor,
\end{equation}
successful iterations and at most
\begin{eqnarray*}
{\cal{I}}_T &= & \left\lfloor \kappa_s\frac{f(x_0)-f_{\mbox{\scriptsize low}}}{\epsilon^{3/2}} \right\rfloor \left(1+\frac{\lvert \rm{log }\gamma_1 \rvert}{\rm{log }\gamma_2}\right)+\frac{1}{\rm{log} \gamma_2}{\rm{log}}\left(\frac{\sigma_{\max}}{\sigma_0}\right) +  \left\lfloor \kappa_u (f(x_0)-f_{\mbox{\scriptsize low}})\right\rfloor,
\end{eqnarray*}
iterations to produce an iterate $x_{\widehat k}$ satisfying (\ref{grad_eps}), with $\kappa_s=\frac{3}{ \eta_1\sigma_{\min}\zeta^{3/2}}$
and $\zeta$ as in Lemma \ref{sk}, and $\kappa_u=\frac{3}{\eta_1\sigma_{\min}}$.
\end{theorem}
\vskip 5pt
\begin{proof}
The mechanism of Algorithm \ref{ARCalgonew}  for updating $\sigma_k$ has the form \req{newsigma1}--\req{newsigma3}.
An unsuccessful iteration in $\mathcal{U}_{k,2}$ does not affect the value of the regularization parameter
as $\sigma_{k+1}=\sigma_k$.  Moreover, the assumptions of  Lemma \ref{Lsigmamax} hold  at iterations $k\in \mathcal{S}_k$.
Hence, $\sigma_k\le \sigma_{\max}, \, \mbox{ for all }k\ge 0$, due to Lemma \ref{Lsigmamax}.

The upper bound on the cardinality $ |\mathcal{S}_k|$  of $\mathcal{S}_k$  follows from \cite[Theorem 2.5]{Toint1}. 
Then, by using (\ref{diffT}) and Lemma  \ref{sk},
at each successful iteration before termination it holds
\begin{eqnarray}f(x_k)-f(x_k+s_k)&\ge& \eta_1(T_2(x_k,0)-T_2(x_k,s_k)) \nonumber \\
& {\ge}&\eta_1 \frac{\sigma_k}{3}\|s_k\|^3 \label{ks_uns} \\
& \ge & \eta_1\frac{\sigma_{\min}}{3} \zeta^{3/2}\| \nabla f(x_k+s_k)\|^{3/2}   \nonumber  \\
&\eqdef&\kappa_s^{-1}\| \nabla f(x_k+s_k)\|^{3/2}.\label{ks} 
\end{eqnarray}
Consequently, before   termination  \eqref{grad_eps} it holds  $f(x_k)-f(x_k+s_k)\ge \kappa_s^{-1}\epsilon^{3/2}$ 
which implies
\[
f(x_0)-f(x_{k+1})=\sum_{j\in\mathcal{S}_k}(f(x_j)-f(x_j+s_j))\ge |\mathcal{S}_k| \kappa_s^{-1}\epsilon^{3/2},
\]
and  (\ref{succfinale}).

The  upper bound on $ |\mathcal{U}_{k,1}|$ follows from \cite[Lemma 2.4]{Toint1}. In particular,
by (\ref{newsigma1})--(\ref{newsigma3})  it holds $\sigma_0\gamma_1^{|\mathcal{S}_k|} \gamma_2^{|\mathcal{U}_{k,1}|} \le \sigma_k$
and   (\ref{sigmamax})  implies
\[
\begin{split}
 |\mathcal{U}_{k,1}|&\le |\mathcal{S}_k| \frac{|\mbox{log} \gamma_1|}{\mbox{log} \gamma_2}  +\frac{1}{\mbox{log} \gamma_2}\mbox{log}\left(\frac{\sigma_{\max}}{\sigma_0}\right) .
\end{split}
\] 

As for $ |\mathcal{U}_{k,2}|$, it  is  less or equal than the number of successful iterations with $\|s_k\|\ge1$. 
By construction, an unsuccessful iteration in $\mathcal{U}_{k,2}$ 
occurs at most once between two successful iterations with the first one such that ${\rm flag}=1$,
and it can not occur between two successful iterations if ${\rm flag}$ is null at   the first of such  iterations.
In fact, ${\rm flag}$ is reassigned only at the end of a successful iteration and  
can be set to one only in case of successful iteration with $\|s_k\|\ge1$, see Step 5 of Algorithm \ref{ARCalgonew},
except for the first iteration. If  the case ${\rm flag}=1$ and   $\|s_k\|<1$  occurs    then 
flag is set  to zero and  is not further changed    until  the subsequent successful iteration. Moreover, as ${\rm flag}$ is initialized to one, at most one additional unsuccessful iteration in $\mathcal{U}_{k,2}$  may occur before the first successful iteration.

Noting that,   by \eqref{ks_uns},
\begin{eqnarray*}
f(x_0)-f(x_{k+1})&=&\sum_{j\in\mathcal{S}_k} \left( f(x_j)-f(x_j+s_j) \right)\\
&\ge&  \sum_{\begin{small}\begin{array}{c} j\in\mathcal{S}_k\\ \|s_k\|\ge 1\end{array}\end{small}}(f(x_j)-f(x_j+s_j))\\
&\ge&  \eta_1 \frac{\sigma_{\min}}{3} \sum_{\begin{small}\begin{array}{c} j\in\mathcal{S}_k\\ \|s_k\|\ge 1\end{array}\end{small}} \|s_k\|^3, 
\end{eqnarray*}
we have
\begin{eqnarray*}
f(x_0)-f_{low}&\ge &  \eta_1 \frac{\sigma_{\min}}{3}  |\mathcal{U}_{k,2}|\eqdef \kappa_u^{-1} |\mathcal{U}_{k,2}|
\end{eqnarray*}
Then,   we obtain $|\mathcal{U}_{k,2}|\le \left\lfloor \kappa_u (f(x_0)-f_{\mbox{\scriptsize low}})\right\rfloor$  
and the proof is concluded.
\end{proof}
\vskip 5pt

The complexity analysis presented above implies 
 \[
\label{liminf}
\liminf_ {k\rightarrow \infty} ||\nabla f(x_k)\|=0.
\]
Further  characterizations of the asymptotic behaviour of $\|\nabla f(x_k)\|$ and $\|s_k\|$ are  given below  where  the sets ${\cal{S}}$, 
$\mathcal{U}_{1}$, $\mathcal{U}_{2}$ are defined as  
\begin{eqnarray*}
{\cal{S}}&=&\{k\ge 0: \, k \,  \mbox{ successful or very successful in the sense of Step 5}\},\\
\mathcal{U}_{1}&=&\{k\ge 0: \, k  \, \mbox{ unsuccessful in the sense of Step 5}\},\\
\mathcal{U}_{2}&=&\{k\ge 0: \, k \,  \mbox{ unsuccessful  in the sense of Step 4}\}.
\end{eqnarray*}

\begin{theorem}\label{fop}
Suppose  that $f$ in \eqref{Pb} is lower bounded by $f_{\mbox{\scriptsize low}}$, and that the 
assumptions of  Theorem \ref{TComplexityExact}  hold.
Then, the steps $s_k$ and the iterates $x_k$ generated by Algorithm \ref{ARCalgonew} satisfy
\begin{equation}\label {limsk}
\|s_k\|\rightarrow 0, \quad \mbox { as } k\rightarrow \infty, \quad  k\in {\cal{S}},
\end{equation}
and 
\begin{equation}\label {limgrad}
\| \nabla f(x_k) \|\rightarrow 0 , \quad \mbox{ as } k\rightarrow \infty.
\end{equation}
Moreover, unsuccessful iterations in $\mathcal{U}_{2}$ do not occur eventually.
\end{theorem}
\begin{proof}
The first claim is proved paralleling \cite[Lemma $5.1$]{ARC1}. In particular, by (\ref{m_T2}) 
\[
\begin{split}
f(x_k)-f(x_{k+1}) & \ge  \eta_1(T_2(x_k,0)-T_2(x_k,s_k)) \ge  \eta_1\frac{\sigma_{\min}}{3}\|s_k\|^3, \, \, k\in {\cal{S}}.  
\end{split}
\]
Since $f$ is lower bounded by $f_{low}$, one has
\[
f(x_0)-f_{low}\ge f(x_0)-f(x_{k+1})= \sum_{j=0, \, j\in S}^k (f(x_j)-f(x_{j+1}))\ge  \eta_1\frac{\sigma_{\min}}{3}\sum_{j=0, \, j\in {\cal{S}}}^k \|s_j\|^3,\quad  k\ge 0,
\]
which implies  convergence of the series  $\sum_{k=0,~k\in\mathcal{S}}^{\infty} \|s_k\|^3$  and the first claim  as a consequence.

As for $\|  \nabla f(x_{k})\|$, Lemma \ref{sk} provides
\[
\zeta\|  \nabla f(x_{k+1})\|\le \| s_k \|^2\rightarrow 0, \quad \textrm{as}~k\rightarrow \infty, ~ k\in\mathcal{S}.
\]
This fact along with $\nabla f(x_{k+1}) =\nabla f(x_{k})$ at unsuccessful iterations provides the convergence of $\{\|  \nabla f(x_{k})\|\}$ to zero.

Finally,  the behaviour of $\{ \|s_k\| \}_{k\in {\cal{S}}}$ implies that  eventually all successful iterations are such that $\|s_k\|<1$. Thus, 
the mechanism of Algorithm \ref{ARCalgonew} gives ${\rm flag}=0$ for all $k$ sufficiently large and unsuccessful iterations  in the sense of Step 4
can not occur.
\end{proof}

\section{Convergence to second order critical point} \label{2ndo}
In this section we focus on the convergence of the sequence  generated by our  procedure to second-order critical points $x^*$:
\[
\nabla f(x^*)=0\quad \textrm{and}\quad \lambda_{\min}(\nabla^2 f(x^*))\ge 0.
\]

First, we   analyze the asymptotic behaviour of  $\{x_k\}$ in the case where the
approximate Hessian  $B_k$ becomes positive definite along a converging subsequence of  $\{x_k\}$. 
In such a context, we show $q$-quadratic convergence of $\{x_k\}$ under an additional mild requirement on the step, namely  the Cauchy condition. 
Second, we consider the case where the model $B_k$ is not convex and obtain a second order complexity bound in accordance with the study of Cartis et al. \cite{CGToint}.
\vskip 5pt
\begin{theorem}\label{sop1}
Suppose  that $f$ in \eqref{Pb} is lower bounded by $f_{\mbox{\scriptsize low}}$, and that the
assumptions of  Theorem \ref{TComplexityExact}  hold.
Suppose  that $\{x_{k_i}\}$  is a subsequence of successful  iterates converging to some $x^*$ and that $B_{k_i}$ are positive definite 
whenever $x_{k_i}$ is sufficiently close to $x^*$.
Then 
\begin{description}
\item{i)} $x_k\rightarrow x^*$ as  $k\rightarrow \infty$ and $x^*$ is second-order critical.
\item{ii)} If   $s_k$ satisfies
\beqn{cauchy1}
m(x_k,s_k,\sigma_k)\le m(x_k,s_k^C,\sigma_k), \, \, \forall k\ge 0 ,
\eeqn
where $s_k^C$ is the Cauchy step, i.e. 
$$
s_k^C=-\alpha_k^C \nabla f(x_k)  \quad \mbox {and} \quad \alpha_k^C=\argmin_{\alpha\ge 0} m_k(x_k,-\alpha \nabla f(x_k),\sigma_k),
$$
then  all the iterations are eventually successful and   $x_k\rightarrow x^*$ $q$-quadratically.
\end{description}
\end{theorem}
\begin{proof} $i)$
From \eqref{bound_cost} and \eqref{limsk},  it follows 
\begin{equation}
\label{asympagree}
\| \nabla^2f(x_{k})-B_{k} \|\le  \max(C,\alpha (\kappa_B+\sigma_{\max})) \| s_k\|\rightarrow 0,\quad k\rightarrow \infty, ~  k\in\mathcal{S}.
\end{equation}
As a consequence, standard perturbation results on the eigenvalues of symmetric matrices and the  
convergence of  $\{x_{k_i}\}$ to $x^*$ give that $ \nabla^2 f(x^*)$ is positive definite.
Thus, $x^*$ is an isolated limit point and the claim $i)$ is completed by using  (\ref{limsk}) and
\cite[Lemma 4.10]{more}.

$ii)$ From the convergence of $\{x_k\}$ to $x^*$, (\ref{asympagree}) and   the positive definiteness of $\nabla^2 f(x^*)$   it follows that
\[
\lambda_{\min}(B_{k})\ge \underline{\lambda} >0,\quad \forall k \in\mathcal{S}  ~\textrm{sufficiently large}.
\]
Moreover,  we know that  unsuccessful iterations in  $\mathcal{U}_{2}$ do not occur eventually. 
Then, taking into account that $B_k$  is not modified along the unsuccessful iterations in  $\mathcal{U}_{1}$,  
we conclude that 
\[
\lambda_{\min}(B_{k})\ge \underline{\lambda},\quad \forall k \,\, \textrm{sufficiently large}.
\]

In order to show that  all the iterations are eventually successful, we start using (\ref{tc}), 
(\ref{derivm}) and  obtain 
\begin{eqnarray*}
\frac{\|s_k\|}{\|(B_k +\sigma_k\| s_k\|I)^{-1} \|}-\|\nabla f(x_k)\|\le   \theta \|\nabla f(x_k)\|.
\end{eqnarray*}
Since 
$$
\|(B_k +\sigma_k\| s_k \|I)^{-1}\|=\frac{1}{\lambda_{\min}(B_k)+\sigma_k\|s_k\|}\le \frac{ 1}{ \underline{\lambda} },
$$
we get 
\begin{equation}
\label{bounds}
\|s_k\|\le    \frac{1+\theta}{\underline{\lambda}} \,  \|\nabla f(x_k)\|, \quad \forall k~\textrm{sufficiently large},
\end{equation}
and  $\|s_k\|\rightarrow 0$ due to  \eqref{limgrad}. Moreover,   by \eqref{m_T2},  \eqref{cauchy1} and  \cite[Lemma 2.1]{ARC1} 
\begin{eqnarray}
T_2(x_k,0)-T_2(x_k,s_k)&\ge&  m(x_k,0,\sigma_k)-m(x_k,s_k,\sigma_k)\nonumber\\
&\ge& \frac{\|\nabla f(x_k)\|}{6\sqrt{2}}  \min \left ( \frac{\|\nabla f(x_k)\|}{1+\|B_k\|}, \frac{1}{2} \sqrt{\frac{\|\nabla f(x_k)\|}{\sigma_k}}\right ), \label{cT}
\end{eqnarray}
and Assumption \ref{ass_Bk} and  Lemma \ref{Lsigmamax} yield
$$
T_2(x_k,0)-T_2(x_k,s_k)\ge  \frac{\|\nabla f(x_k)\|}{6\sqrt{2}}
\min \left ( \frac{\|\nabla f(x_k)\|}{1+\kappa_B}, \frac{1}{2} \sqrt{\frac{\|\nabla f(x_k)\|}{\sigma_{\max}}}\right ).
$$
Thus, eventually \eqref{limgrad} and  \eqref{bounds} give 
\begin{eqnarray*}
T_2(x_k,0)-T_2(x_k,s_k)&\ge&  \frac{\|\nabla f(x_k)\|^2}{6\sqrt{2}(1+\kappa_B)}\ge \frac{\underline{\lambda}^2}{6\sqrt{2}(1+\kappa_B)(1+\theta)^2}\|s_k\|^2\eqdef \kappa_c \|s_k\|^2 ,
\end{eqnarray*}
and  by \eqref{Ekcases}  and (\ref{fsppbound2})
$$
1-\rho_k   =   \frac{   f(x_k+s_k)-T_2(x_k,s_k) }{ T_2(x_k,0)-T_2(x_k,s_k) }  
\le \frac{E_k(s_k)}{ \kappa_c\|s_k\|^2    } <   \frac{(\alpha(\kappa_B+\sigma_{\max})+L/3)\|s_k\|^3}{2 \kappa_c\|s_k\|^2 },
$$
i.e., $\rho_k \rightarrow 1$ and the iterations are very successful eventually.

Finally, \eqref{bounds} and  Lemma \ref{sk} provide
\[
\|\nabla f(x_{k+1})\|\le \frac{\| s_k \| ^2}{\zeta} \le \frac{(1+\theta)^2}{\zeta \underline{\lambda}^2}\|\nabla f(x_k)\|^2, \quad \forall k  ~\textrm{sufficiently large},
\]
and the $q$-quadratic convergence of the sequence $\{x_k\}$ follows in a standard way  by means of the Taylor's expansion. 
\end{proof}

\vskip 5pt

Dropping the assumption that $B_k$ is  positive definite, 
convergence to second order critical points can be studied. 
Following \cite{CGToint} where a modification   of the ARC algorithm in \cite{ARC1} is proposed, we
equip  Algorithm \ref{ARCalgonew} with a further stopping criterion and impose an additional condition on the step.
First, Algorithm \ref{ARCalgonew} is stopped  when 
\begin{equation}
\label{tc.2}
\| \nabla f(x_k)\| \le \epsilon \quad \textrm{and}\quad \lambda_{\min}(B_k)\ge- \epsilon_H,\quad \epsilon,~\epsilon_H>0,
\end{equation}
which represents the approximate counterpart of the second-order optimality conditions 
with the Hessian matrix approximated by $B_k$. The above criterion does not imply,
in general, vicinity to local minima, as well as it does not guarantee the iterates to be distant from saddle points.
Then, the possibility of referring to the strict-saddle property \cite{Lee} may play a significant role;  
indeed \eqref{tc.2} implies closeness 
to a local minimum for sufficiently small values of the tolerances $\epsilon$ and $\epsilon_H$.

Second, the trial step $s_k$ computed in Step $2.2$ of Algorithm \ref{ARCalgonew} is required to 
satisfy the following additional condition: if $B_k$ is not positive semidefinite, then 
\begin{equation}\label{curvature}
m(x_k,s_k,\sigma_k )\le m(x_k,s_k^E,\sigma_k),
\end{equation}
where $s_k^E$ is defined as
\begin{equation}\label{eig}
s_k^E=\alpha_k^E u_k \quad \mbox{and} \quad \alpha_k^E=\argmin_{\alpha \ge 0} m_k(\alpha u_k), 
\end{equation}
and $u_k$ is an approximation of the eigenvector of $B_k$ associated with its smallest eigenvalue $\lambda_{\min}(B_k)$, in the sense that
\begin{equation}
\label{Cond_eig}
\nabla f(x_k)^T u_k\le 0\quad \textrm{and}\quad u_k^T B_ku_k\le \kappa_{\footnotesize \mbox{snc}}\lambda_{\min}(B_k)\|u_k\|^2,
\end{equation}
for some constant $k_{\footnotesize \mbox{snc}}\in(0,1]$.
Note that the minimization in \eqref{eig} is    global which implies
\begin{eqnarray}
& & \nabla f(x_k)^T s_k^E+(s_k^E)^T B_ks_k^E+\sigma_k\|s_k^E\|^3=0, \label{Condmg_1}\\
& & 
(s_k^E)^T B_ks_k^E+\sigma_k\|s_k^E\|^3\ge 0.\label{Condmg_2}
\end{eqnarray}

We refer to the resulting  algorithm as ARC Second Order critical point (\metodop).
The termination criterion adopted here does not affect the 
mechanism for updating $\sigma_k$, then the upper bound  $\sigma_{\max}$ on $\sigma_k$ 
given in Lemma \ref{Lsigmamax} is still valid.

Let  $\widetilde{\cal{S}}_k$ denote the set of indices of successful iterations of  \metodo whenever $\| \nabla f(x_k)\| > \epsilon$ and/or
$\lambda_{\min}(B_k)<- \epsilon_H$, i.e.,  the indices of successful iterations before \eqref{tc.2} is  met. 
Following \cite{CGToint} we  also let  $\widetilde{\cal{S}}_k^{(1)}$ be the set of indices  of successful iterations where $\| \nabla f(x_k)\| > \epsilon$  
and $\widetilde{\cal{S}}_k^{(2)}$ be the set of indices of successful iterations where $\lambda_{\min}(B_k)<- \epsilon_H$. 
Let  $\widetilde{{\cal{U}}}_{k,1}$ and  $\widetilde{{\cal{U}}}_{k,2}$ denote the set of unsuccessful iterations of \metodo analogously
to \eqref{uk1} and \eqref{uk2}. 
Remarkably, the cardinality of both  $\widetilde{\cal{S}}_k^{(1)}$  and  $\widetilde{{\cal{U}}}_{k,2}$
is   the same as in Algorithm  \ref{ARCalgonew}, see Theorem \ref{TComplexityExact}, while proceeding as in Theorem 
\ref{TComplexityExact} the cardinality of $\widetilde{{\cal{U}}}_{k,1}$ is bounded in terms of the number of successful iterations 
$\widetilde{\cal{S}}_k$, see also \cite[Lemma 2.6]{CGToint}. 
Hence, it remains to derive the cardinality of $\widetilde{\cal{S}}_k^{(2)}$.

\vskip 5pt
\begin{lemma}
Suppose that $f$ in \eqref{Pb} is lower bounded by $f_{low}$ and the  assumptions of  Theorem \ref{TComplexityExact}  hold. 
Suppose that $s_k$ satisfies \eqref{curvature}. Then, the number of successful iterations of  Algorithm \metodo  with $\lambda_{\min}(B_k)<- \epsilon_H$ is bounded above by 
$$
\left\lfloor \kappa_e\frac{f(x_0)-f_{\mbox{\scriptsize low}}}{\epsilon_H^3} \right\rfloor,
$$
where
$\kappa_e=\frac{6\sigma_{\max}^2}{\eta_1 \kappa_{\footnotesize \rm{snc}}^3}.
$

\end{lemma}
\begin{proof}
The proof parallels that of \cite[Lemma 2.8]{CGToint}.
We have 
\begin{eqnarray}f(x_k)-f(x_k+s_k)&\ge& \eta_1(T_2(x_k,0)-T_2(x_k,s_k)) \nonumber \\
& = & \eta_1( m(x_k,0,\sigma_k )-m(x_k,s_k,\sigma_k )+\frac{\sigma_k}{3}\|s_k\|^3) \nonumber \\
& \ge & \eta_1(m(x_k,0,\sigma_k )-m(x_k,s_k^E,\sigma_k )) \nonumber \\
& \ge & \eta_1  \frac{\sigma_k}{6}\|s_k^E\|^3\nonumber \\
& \ge & \eta_1  \frac{- \kappa_{\footnotesize \mbox{snc}}^3 \lambda_{\min}(B_k)^3}{6\sigma_{\max}^2}\\
& \ge &\eta_1 \frac{ \kappa_{\footnotesize \mbox{snc}}^3\epsilon_H^3}{6\sigma_{\max}^2}\nonumber
\end{eqnarray}
in which we have used   \eqref{m_T2}, \eqref{curvature},  \eqref{Condmg_1}, \eqref{Condmg_2} and \eqref{Cond_eig}. 
As a consequence, letting $\kappa_e$ as in the statement of the theorem, before termination it holds
\[
f(x_0)-f_{low} \ge f(x_0)-f(x_{k+1})\ge\sum_{j\in\widetilde{\cal{S}}_k^{(2)}}(f(x_j)-f(x_j+s_j))\ge |\widetilde{\cal{S}}_k^{(2)}| \kappa_e^{-1}\epsilon_H^3,
\]
and the claim follows.
\end{proof}
\vskip 5pt
We thus conclude that  Algorithm \metodo produces an iterate $x_{\widehat{k}}$ satisfying \eqref{tc.2} within at most
$$
O\left(\mbox{max}(\epsilon^{-3/2},\epsilon_H^{-3})\right) ,
$$ 
iterations,  in accordance with the  complexity result in  \cite{CGToint}.

\section{Finite sum minimization}\label{fs}
Large-scale instances of the finite-sum problem (\ref{finite_sum}) can be conveniently solved by subsampled procedures
where $\nabla f^2(x_k)$ is approximated  by randomly sampling component functions $\phi_i$ 
\cite{review}. The resulting approximation of $\nabla f^2(x_k)$ takes the form 
\beqn{subhessian}   
\nabla^2 f_{{\cal D}_k} (x_k)= \frac{1}{|{\cal D}_k|} \sum_{i \in {\cal D}_k} \nabla^2 \phi_i(x_k), 
\eeqn
with  $ {\cal D}_k \subset \{1,2,\ldots,N\}$  and $|{\cal D}_k|$ being the so-called sample size.

We discuss the application of Algorithm \ref{ARCalgonew} to problem (\ref{finite_sum})  with 
\beqn{Bk_fs}  
B_k=\nabla^2 f_{{\cal D}_k} (x_k),
\eeqn
giving both deterministic and probabilistic results. The application of Algorithm \ref{ARCalgonew} to problem (\ref{finite_sum})  with such Hessian approximation
is supported  by results in the literature which give the sample size required  to obtain  $B_k$ satisfying condition  (\ref{Dk})   in probability  and will be addressed below.

Let us make the following assumption on the objective function.

\vskip 5pt 
\begin{ipotesi}\label{ass_phi}
Suppose that, for any $x\in\R^n$, there exist non-negative upper bounds $\kappa_{\phi}(x)$ such that
\[
\max_{i\in\{1,...,N\}}\| \nabla^2\phi_i(x)\| \le \kappa_{\phi}(x).
\]

\end{ipotesi}

\vskip 5 pt
Uniform and non-uniform sampling  strategies  have been proposed \cite{review, Cin, kl, Roosta, Roosta_2p}; 
for instance, the  following Lemma provides the size of uniform sampling  which probabilistically  satisfies  (\ref{Dk}).
\vskip 5pt
\begin{lemma} \label{dkbounds}
Assume that Assumption  \ref{ass_phi} holds,  $C_k>0$ is given,  the subsample ${\cal D}_k$ is chosen randomly and uniformly from $\{1,2,\ldots,N\}$ and  
$B_k$ is as in (\ref{Bk_fs}).  Then,  given $\bar \delta \in (0,1) $, 
\begin{equation}
\label{err_hess}
Pr(\lvert\lvert \nabla^2 f(x_k)-B_k \rvert\rvert \le C_k)\ge 1-\bar{\delta},
\end{equation}
whenever  the cardinality  $|{\cal D}_k|$ of  ${\cal D}_k$ satisfies 
\beqn{minbound}
|{\cal D}_k|\ge \min\left \{ N,\left\lceil\frac{4\kappa_{\phi}(x_k)}{C_k}
  \left(\frac{2\kappa_{\phi}(x_k)}{C_k}+\frac{1}{3}\right)
  \,\log\left(\frac{2n}{\overline{\delta}}\right)\right\rceil\right \}\eeqn
  
\end{lemma}
\begin{proof} See \cite[Theorem 7.2]{bmgt}.
\end{proof}
\vskip 5pt
We hereafter assume the existence of $\overline{\kappa}_\phi\ge 0$ such that
\beqn{barkphi}
\sup_{x\in\R^n} \kappa_\phi(x)\le \overline{\kappa}_\phi,
\eeqn
yielding  $\sup_{x\in \R^n}\|\nabla^2 f(x)\|\le \overline{\kappa}_{\phi}$ and Assumption \ref{ass_Bk} with $ \kappa_B=\overline{\kappa}_{\phi}$.

We first give  deterministic results, namely   properties which are valid independently 
from Assumption  \ref{ass_Dk}  on $B_k$, now   guaranteed with probability $1-\bar \delta$ by Lemma \ref{dkbounds}. 
In the following theorem the only requirement on $B_k$ is the boundness of its norm,  i.e. Assumption  \ref{ass_Bk}; concerning the trial step $s_k$,
the Cauchy condition  \eqref{cauchy1} is assumed.\footnote {This result is valid independently  from the specific form of $f$  considered in this section, provided that the norm of the Hessian of $f$ is bounded in an open convex set containing all the sequence $\{x_k\}$ and 
Assumptions \ref{ass_Bk} holds.}
\vskip 5pt
\begin{theorem}
Let $f\in C^2(\mathbb{R}^n)$. Suppose  that $f$ in \eqref{Pb} is lower bounded by $f_{\mbox{\scriptsize low}}$, Assumption
\ref{ass_phi}, conditions \eqref{cauchy1} and  \eqref{barkphi}  hold.   Then,
\begin{itemize}
\item[i)] Given   $\epsilon>0$, Algorithm \ref{ARCalgonew} takes at most   $O(\epsilon^{-2})$ successful  iterations to satisfy $\|\nabla f(x_k)\|<\epsilon$.  
\item[ii)] $\| \nabla f(x_k) \|\rightarrow 0$, as $k\rightarrow \infty$ and therefore all the accumulation points of the sequence $\{x_{k}\}$, if any, are first-order stationary points.
\end{itemize}
 \begin{itemize}
\item[iii)]  If  $\{x_{k_i}\}$  is a subsequence of iterates converging to some $x^*$ such that $\nabla f^2(x^*)$ is definite positive, then $x_k\rightarrow x^*$ 
as $k\rightarrow \infty$.
\end{itemize}
\end{theorem}
\begin{proof} 
$i)$. The claim follows from Lemma 3.1--3.3 and Corollary 3.4 in \cite{ARC2}. In fact,
despite the acceptance criterion in \cite{ARC2} is \eqref{pi} instead of   \eqref{rho},  
we can rely on the proof of \cite[Lemma 3.2]{ARC2} thanks to (\ref{cT}) and   considering that 
{
\[
f(x_k+s_k)-T_2(x_k,s_k)\le 2\overline{\kappa}_{\phi}\| s_k \|^2, \quad  k\ge 0.
\]}

$ii)$ The sub-optimal complexity result in Item $i)$ guarantees that $\liminf_{k\rightarrow \infty} \|\nabla f(x_k)\|=0$ and that
the number of successful iterations is not finite. 
Moreover, $\lim_{k\rightarrow \infty} \|\nabla f(x_k)\|=0$ follows by Assumption \ref{ass_phi} {, \eqref{barkphi}}
and \cite[Corollary $2.6$]{ARC1}.

$iii)$ Proceeding as in Theorem \ref{fop} we obtain  \eqref{limsk}. 
Since $\nabla^2 f(x^*)$ in positive definite, $x^*$ is an isolated limit point;  consequently,  \eqref{limsk} and 
Lemma  \cite[Lemma 4.10]{more} yield the claim.
\end{proof}
\vskip 5pt

Focusing on the optimal complexity result, 
 we observe that Algorithm  \ref{ARCalgonew} requires at most $O(\epsilon^{-3/2})$ iterations to satisfy  
$\|\nabla f(x_k)\|\le \epsilon$ with probability  $1-\delta$, $\delta \in (0,1)$, provided that  the sample size is chosen accordingly to  \eqref{minbound}  and  
$\bar\delta$ is suitable chosen.    
In fact, let $\mathcal{E}_i $ be the event: ``the relation $\|  \nabla^2 f(x_i)-B_i \| \le C_i$ holds at iteration $i$, $1\le i\le k$'',
and $\mathcal{E}(k)$ be the event: ``the relation $\|  \nabla^2 f(x_i)-B_i \| \le C_i$ holds  for the entire $k$ iterations''.
If  the events $\mathcal{E}_i $ are independent,  then 
due to \eqref{err_hess} 
\[
Pr(\mathcal{E}(k))\equiv Pr\left(\bigcap_{i=1}^k \mathcal{E}_i\right)=(1-\bar{\delta})^k.
\]
Thus, requiring  that the event $\mathcal{E}(k)$ occurs with   probability  $1-\delta$, we obtain
\[
Pr(\mathcal{E}(k))=(1-\bar{\delta})^k=1-\delta, \quad \mbox{i.e.,} \quad \bar{\delta}=1-\sqrt[k]{1-\delta}=O\left(\frac{\delta}{k}\right).
\]
Taking into account the iteration complexity, we set  $ k=O\bigl(\epsilon^{-3/2}\bigr)$ and  deduce the following choice of $\bar \delta$: 
\begin{equation}
\label{delta}
\bar{\delta}=O(\delta \epsilon^{3/2}).
\end{equation}
Summarizing, choosing, at each iteration,  $\bar{\delta}$ according to \eqref{delta} and the sample size according to \eqref{minbound}, 
the complexity result in Theorem \ref{TComplexityExact} holds with probability of success $1-\delta$.
We underline that the  resulting  per-iteration failure probability $\bar \delta$  is not too demanding  in what concerns the sample size, 
because it influences only the logarithmic factor   in \eqref{minbound}, see \cite{Roosta}.  

  Observe that   \eqref{minbound} and  $C_k=\alpha (1-\theta)\|\nabla f(x_k)\|$ yield $ |{{\cal D}_k}| =O(\|\nabla f(x_k)\|^{-2})$ 
as long as $N$ is large enough so that  full sample size is not reached. Hence, in the general case, 
$|{\cal D}_k|$ is expected to grow along the iteration and reach values of order $\epsilon^{-2}$ at termination.

In the specific case where  $k \in {\cal S}\cup {\cal U}_{1}$,  $C_k=\alpha (1-\theta)\|\nabla f(x_k)\|$
and $\lambda_{\min}(B_k)\ge \underline \lambda$  for some positive $\underline \lambda$, 
using \eqref{bounds} and Lemma \ref{sk} we obtain
$$
\|\nabla f(x_k)\| \ge \frac{\underline \lambda}{1+\theta} \|s_k\| \ge \frac{\sqrt{\zeta} \underline \lambda}{1+\theta} \sqrt {\|\nabla f(x_{k+1})\|}.
$$
Then $\|\nabla f(x_k)\| \ge \frac{\sqrt{\zeta\epsilon} \underline \lambda}{1+\theta} $, provided that the algorithm does not terminate at iteration $k+1$; consequently,
$|{\cal D}_k|$ is expected to grow along such iterations and reach values of order $\epsilon^{-1}$ eventually. 
On the other hand,  the sample size for Hessian approximation is expected to be small with respect to $O(\epsilon^{-1})$ when $C_k$ is set equal to the arbitrar constant accuracy $C$, hence the iterations at which $C_k=C$ can be neglected within this analysis. Taking into account that  ${\cal U}_{2}$ does not depend on $\epsilon$ we can claim  that,
with probability 
$1-\delta$, at most  $O(\epsilon^{-5/2})$  $\nabla^2 \phi_i$-evaluations are required to compute an $\epsilon$-approximate first order point, provided that $\lambda_{\min}(B_k)\ge \underline \lambda$  at all iterations where  $C_k=\alpha (1-\theta)\|\nabla f(x_k)\|$. 
This is ensured for the subclass of problems where  functions $\phi_i$ are strongly convex. Problems of this type arise, for instance, in classification procedures. For this subclass of problems,
Theorem \ref{sop1}, Item $ii)$ also ensures that, for $k$ sufficiently large, say $k\ge \bar k$, 
with probability $(1-\bar\delta)^{k_o}$  there exists $M>0$ such that 
$$
\|x_{k+1}-x^*\|\le M \|x_k-x^*\|^2,\quad k=\bar k, \ldots, \bar k+k_o-1,
$$
where $x^*$ the unique minimizer.
Specifically, proceeding as in \cite[Theorem 2]{Roosta_in} and   denoting with  $\mathcal{E}_i $ 
the event: ``the relation $\|\nabla^2 f(x_i)-B_i \| \le C_i$ holds at iteration $i$, $i\ge  \bar k$'', 
we have that the overall success   probability in consecutive $k_o$ iterations is 
$$
 Pr\left(\bigcap_{i=\bar k}^{\bar k+k_o-1} \mathcal{E}_i\right)=(1-\bar \delta)^{k_o},
$$ 
which concludes our argument.

\section{Related work}\label{cfr}
Variants of ARC based on suitable approximations of the gradient and/or the Hessian of $f$ have been 
discussed in a few recent lines of work reviewed in this section. 
Besides the algorithm in  \cite{ARC1, ARC2, CGToint},  which
employs approximations for  the Hessian  and is suited for  a generic nonconvex function $f$,
works \cite{Cin, CS, kl, Roosta,Roosta_inexact} propose  variants of the algorithm  given in \cite{ARC1} 
where the gradient and/or the Hessian approximations 
can be performed via subsampling techniques \cite{bbn, review} and are applicable  to the relevant 
class of large-scale finite-sum minimization (\ref{finite_sum}) arising in machine learning; probabilistic/stochastic complexity and 
convergence analysis is carried out.  

Cartis et al. \cite{ARC1, ARC2, CGToint} analyze ARC framework under varying assumptions on the Hessian approximation $B_k$ 
and establish optimal and sub-optimal worst-case iteration bounds for first- and second-order optimality.
First-order complexity was shown to be of $O(\epsilon^{-2})$ iterations under Assumption
\ref{ass_Bk} and, as mentioned in Section \ref{ARCbase}, of $O(\epsilon^{-3/2})$ iterations  when, in addition, $B_k$
resembles the true Hessian and condition  (\ref{AM4}) is satisfied.

Kohler et al. \cite{kl} propose and study a variant of ARC algorithm suited for finite-sum minimization not necessarily convex.
A subsampling scheme for the gradient and the Hessian of $f$ is applied while maintaining first-order complexity
of $O(\epsilon^{-3/2})$  iterations.  The sampling scheme provided guarantees that  the subsampled gradient $g(x_k)$ satisfies
\begin{equation}\label{grad_kholer}
\|\nabla f(x_k)-g(x_k)\|\le M\|s_k\|^2, \quad \forall k\ge 0, \, M>0,
\end{equation}
with prefixed probability, and thesubsampled Hessian $B_k$ satisfies condition  (\ref{AM4}) with prefixed probability. 
As  specified in Section \ref{ARCbase},  condition \req{AM4} is enforced via \req{AMLK} and since the steplength
can be determined only after $g(x_k)$ and $B_k$ are formed, the steplength at the previous iteration is taken

Cartis and Scheinberg \cite{CS} analyze a probabilistic cubic regularization variant where conditions \eqref{grad_kholer} and  (\ref{AM4}) are satisfied with sufficiently high probability.
Enforcing such conditions in a practical setting  calls for an (inner) iterative process which requires 
a step computation at each repetition;  in the worst-case  derivatives accuracy may reach order $O(\epsilon)$ at each iteration (see also \cite{bmgt}).

 As mentioned in Section \ref{ARCbase}, Xu et al. \cite{Roosta} develop and study  a version of ARC algorithm  where a major modification on the level of resemblance 
between $\nabla^2 f(x_k)$ and $B_k$ is made over (\ref{AM4}).
Matrix $B_k$ is supposed to satisfy Assumption \ref{ass_Bk} and 
\begin{equation}\label{accuracyRoosta}
\|(\nabla^2 f(x_k)-B_k)s_k\|\le \mu \|s_k\|, \quad \mu \in (0,1),
\end{equation}
and the latter condition can be enforced building $B_k$ such that $\| \nabla^2 f(x_k)-B_k \|\le \mu  $.
Non convex finite-sum minimization is the motivating application for the proposal, and uniform and non-uniform sampling strategies are provided
to construct matrices $B_k$ satisfying $\| \nabla^2 f(x_k)-B_k \|\le \mu  $ with prefixed probability.
In particular, unlike the rule in \cite{kl}, 
the rule for choosing the sample size at iteration $k$ does not depend on the step $s_k$  which is not available when $B_k$ has to be built.
Worst-case iteration count of order $\epsilon^{-3/2}$ is shown when  $\mu=O(\epsilon)$, while sub-optimal 
worst-case iteration count of order $\epsilon^{-2}$ is achieved if $\mu=O(\sqrt{\epsilon})$.
Note that the accuracy requirement
on $B_k$ is fixed along the iterations   and depends on the accuracy requirement on the gradient's norm, that is on the gradient's norm at the final iteration.  
Then, when  the  Hessian  of problem (\ref{finite_sum}) is approximated  via subsampling  with accuracy $\epsilon$, 
$O(\epsilon^{-2})$  evaluations of matrices  $\nabla^2 \phi_i$ are needed at each iteration,
assuming $N$ sufficiently large.
Additionally, the use of approximate gradient via subsampling 
is addressed in \cite{Roosta_inexact}.

Chen et al. \cite{Cin} propose an  ARC procedure for convex optimization via random sampling. 
Function $f$ is convex and defined as  finite-sum (\ref{finite_sum}) of possibly nonconvex functions. 
Semidefinite positive subsampled approximations $B_k$ satisfying 
$\| \nabla^2 f(x_k)-B_k \|\le \mu_k$ , $  \mu_k \in (0,1)$, are built with a prefixed probability.
Iteration complexity of order $O( \epsilon^{-1/3})$ is proved with respect to the fulfillment of  
condition $f(x_k)-f(x^*)\le \epsilon$, 
$x^*$ being the global minimum of (\ref{finite_sum}); the 
scalar $\mu_{k}$ is updated as $\mu_{k+1}=O(\min(\mu_k, \|\nabla f(x_k)\|))$, 
and the model $m(x_k, s, \sigma_k)$ is minimized on a subspace of $\mathbb{R}^n$ imposing the strict condition 
$\| \nabla_s m(x_k,s_k,\sigma_k)\| \le \theta\min( \|\nabla f(x_k)\|, \|\nabla f(x_k)\|^3, \|s_k\|^2)$, $\theta\in (0,1)$.

Summarizing, our proposal differs from the above works in the following respects.  
In  \cite{kl} the upper bound in (\ref{AMLK})  is replaced by a bound computed using 
information from the previous iteration and no check on the fulfillment of \req{AMLK} is made, while in  \cite{CS}  the error in Hessian approximation is dynamically reduced  to fulfill \req{AMLK};
on the contrary our  accuracy requirement $C_k$ is computable and  condition \req{AM4} is satisfied at every successful iteration
and at any unsuccessful iteration detected in Step 5 without deteriorating computational complexity.
Our proposal  improves  upon \cite{Roosta, Roosta_inexact} in the construction of $B_k$
as the  level of resemblance between $\nabla^2 f(x_k)$ and $B_k$ is not maintained fixed along iterations but adaptively chosen, remaining less stringent than the first-order $\epsilon$ tolerance when the constant accuracy $C$ is selected by the adaptive procedure or, otherwise, whether the current gradient's norm is sufficiently high (see, e.g., \eqref{Dk}--\eqref{Ckbound}); 
it improves  upon \cite{Cin}   as  the prescribed accuracy on $B_k$ (and the sample size) may reduce at some iteration,
the ultimate accuracy on $\| \nabla_s m(x_k,s_k,\sigma_k)\|$  is milder,  
and  our complexity results  are optimal for nonconvex problems while the analysis in \cite{Cin} is limited to convex problems.

\section{Numerical results}\label{numerical}
 In this section we present  the performance of our ARC Algorithm \ref{ARCalgonew} and show that it can be computationally more convenient than ARC variants in the literature. 
Our numerical validation is based on  inexact Hessians built via  uniform subsampling and 
rule (\ref{minbound})  for choosing  the sample size. The results obtained indicate that
suitable levels of accuracy in Hessian approximation and careful adaptations of rule  (\ref{minbound})   improve efficiency of  existing procedures exploiting subsampled Hessians.}
Experiments  are performed on  
nonconvex finite-sum problems arising within the framework of binary classification. 

Given the training data $\{a_i,y_i\}_{i=1}^N$ where $a_i\in\mathbb{R}^d$ 
and $y_i\in\{0,1\}$ represent the $i$-th feature vector and label respectively, we minimize
the empirical risk using a least-squares loss $f$ with sigmoid function. The minimization problems then takes the form:
\begin{equation}
\label{minloss}
\min_{x\in\mathbb{R}^d} f(x)= \min_{x\in\mathbb{R}^d}  \frac{1}{N}\sum_{i=1}^N{\phi_i(x)}= \min_{x\in\mathbb{R}^d}\frac 1 N \sum_{i=1}^N \left( y_i-\sigma\left(a_i^T x\right) \right)^2,
\end{equation}
with the sigmoid function
\[
\sigma(z)=\frac{1}{1+e^{-z}},\qquad z\in\R,
\]
used as a model for predicting the values of the labels. 
The gradient and the Hessian of the component functions $\phi_i(x)$, $i\in\{1,...,N\}$, in \eqref{minloss} take the form: 
\begin{eqnarray}
&&\nabla \phi_i(x)=-2 e^{-a_i^Tx}\left(1+e^{-a_i^Tx}\right)^{-2}\left(y_i-\left(1+e^{-a_i^Tx}\right)^{-1}\right)a_i,\label{der1phi}\\
& & \nabla^2 \phi_i(x)=-2 e^{-a_i^Tx}\left(1+e^{-a_i^Tx}\right)^{-4}\left(y_i\left(\left(e^{-a_i^Tx}\right)^2-1\right)+1-2e^{-a_i^Tx}\right)a_ia_i^T.\label{der2phi}
\end{eqnarray}
Problem \eqref{minloss} can be seen as a neural network without hidden layers and  zero bias and we  refer to $f$ as the training loss. Trivially it has form (\ref{finite_sum}) with $n=d$

For each dataset, a number $N_T$ of testing data $\{\bar a_i,\bar y_i\}_{i=1}^{N_T}$ is used to validate the computed model
and the testing loss measured as 
$ 
\displaystyle\frac{1}{N_T} \sum_{i=1}^{N_T} \left(\bar y_i-\sigma\left(\bar a_i^T x\right) \right)^2.
$

Implementation issues concerning the considered  procedures are introduced in Section \ref{Simpl}. In Sections  \ref{sint} -\ref{real} we give statistics of our runs. 
We test different ARC variants and rules for choosing the sample size of Inexact Hessians and we  perform two sets of experiments. 
First, in Section  \ref{sint}  we compare ARC variants with optimal complexity  on  a set of synthetic datasets  from \cite{bbn}.
Algorithm \ref{ARCalgonew} is compared with versions of ARC  employing: ({\it i}) exact Hessians;  ({\it ii})
inexact Hessians $B_k$ with  accuracy requirement (\ref{diffck}) and $C_k=\epsilon$, $\forall k\ge 0$ \cite{Roosta, Roosta_2p}; 
({\it iii}) inexact Hessians $B_k$ with accuracy requirement  \eqref{AMLK}
implemented as suggested in \cite{kl}, i.e., the unavailable information  $\|s_k\|$ on the right-hand side is replaced with $\|s_{k-1}\|$, for $k>0$.
 \noindent
Second,  in Section \ref{real} we compare a suboptimal variant of our adaptive strategy
with ARC procedure where inexact Hessians are built using a fixed small sample size.
This experiments are motivated by pervasiveness of  prefixed small  sample sizes  
in practical implementations. In fact, inequality \eqref{minbound} yields to full sample 
when high accuracy is imposed, i.e.  when  $C_k$ is sufficiently small, and 
sample sizes $|{\cal D}_k|$ equal to a  prefixed fraction of $N$ are often employed in literature even though 
first-order  complexity becomes  $O(\epsilon^{-2})$ \cite{bkkj, bkm, bbn, Roosta_2p}.


\subsection{ Implementation issues}\label{Simpl}

The implementation of the main phases of ARC variants is given in this section.

The cubic regularization parameter is initialized as $\sigma_0=10^{-1}$ and its minimum value is $\sigma_{\min}=10^{-5}$. 
The parameters $\eta_1$, $\eta_2$, $\gamma_1$, $\gamma_2$, $\gamma_3$ and $\alpha$  are fixed as
\[
\eta_1=0.1,\quad \eta_2=0.8,\quad\gamma_1=0.5,\quad \gamma_2=1.5,\quad \gamma_3=2,\quad \alpha=0.1,
\]
while the failure probability $\overline{\delta}$ in \eqref{err_hess} is set equal to $0.2$. The initial guess is the zero vector $x_0=(0,...,0)^T\in\mathbb{R}^d$ in all runs. 

The minimization of the cubic model in Step $3$ of Algorithm \ref{ARCalgonew} is performed by 
the Barzilai-Borwein gradient method \cite{bb} combined with a nonmonotone linesearch following the proposal in \cite{blms}.
The major per iteration cost of such Barzilai-Borwein process is one Hessian-vector product, needed to compute the gradient of the cubic model. 
The threshold used in the termination criterion (\ref{tc}) is $\theta_k=0.5,\, k\ge 0$.

As for the termination criteria for ARC methods, we imposed a maximum of $500$
iterations and we declared a successful termination  when one of the two following conditions is met:
\[
\|\nabla f(x_k)\|\le \epsilon,\quad |f(x_k)-f(x_{k-1})|\le 10^{-6}|f(x_k)|,\quad \epsilon=10^{-3}.
\]

In order to measure  the computational cost, as in \cite{bbn} we use the number of Effective Gradient Evaluations (EGE), that is the sum of function and Hessian-vector product evaluations. This is a pertinent measure since the major cost in the evaluation of each component function $\phi_i$, $1\le i\le N$, at $x\in\mathbb{R}^d$ consists in the computation of the scalar product $a_i^T x$. Once evaluated, this scalar product can be reused for obtaining $\nabla \phi_i(x)$, while the computation of $\nabla^2 \phi_i(x)$ times a vector $v\in\mathbb{R}^d$ requires the scalar product $a_i^Tv$ and it is as expensive as  one $\phi_i(x)$ evaluation (see \eqref{der1phi}--\eqref{der2phi}). Consequently, each full Hessian-vector product costs as one function or gradient evaluation. 
When $|\mathcal{D}_k|$ samples are used for the Hessian approximation $B_k$, the cost of one matrix-vector product of the form $B_k v$ is counted as $ |\mathcal{D}_k|/N$ EGE. 

 The algorithms were implemented in Fortran language and run on an Intel Core i5, $1.8$ GHz $\times~1$ CPU, $8$ GB RAM.

\subsection{Synthetic datasets}\label{SecSynthetic}\label{sint}
The first class of databases we consider is a set of synthetic datasets  from \cite{bbn}, firstly proposed in \cite{mcfs2013}.
These datasets have been constructed so that    Hessians have condition numbers of  order up to $10^7$ and a wide spectrum of the 
eigeinvalues  and allow testing on   moderately  ill-conditioned problems.  
We  scaled them,  in order to have entries in the interval $[0,1]$, as follows. 
Let $D\in \Re^{(N+N_T)\times d}$ be the matrix containing the training and testing features of the original dataset, that is 
\[
e_i^T D=a_i^T \ \mbox{ for } i\in\{1,...,N\},\quad e_{N+i}^T D=\bar a_i^T \ \mbox{ for } i\in\{1,...,N_T\},
\]
and let $m_j=\min_{i\in\{1,...,N+N_T\}}D_{ij}$ and $M_j=\max_{i\in\{1,...,N+N_T\}}D_{ij}$, for $j=1,\ldots,d$.
Then, matrix $D$ is scaled as
\[
{D}_{ij}\eqdef\frac{D_{ij}-m_j}{M_j-m_j},\ \mbox{ for } i\in\{1,...,N+N_T\}, \, \,  j\in\{1,\ldots,d\}.
\] 

The computation of the matrix $B_k$ accordingly to \eqref{minbound} involves the constant
\[
\kappa_{\phi}(x_k)=\max_{i\in\{1,...,N\}}\left\{2e^{-a_i^Tx_{k}}\left(1+e^{-a_i^Tx_{k}}\right)^{-4}\left|y_i\left(\left(e^{-a_i^Tx_{k}}\right)^2-1\right)+1-2e^{-a_i^Tx_{k}}\right| \|a_i\|^2\right\}.
\]
Since the values $a_i^Tx_k$, $1\le i\le N$, are available from the exact computation of $f(x_k)$, we evaluated $\kappa_{\phi}(x_k)$
at the (offline) extra cost of computing $\|a_i\|^2$, $1\le i\le N$. 

In our implementation of Algorithm  \ref{ARCalgonew} the value of $C$ used in \eqref{bound1}  whenever $\|s_k\|\ge 1$  is such that  $|\mathcal{D}_0|$ computed via \eqref{minbound} with $C_0=C$ satisfies $|\mathcal{D}_0|/N=0.1$. 
We shall hereafter refer to the implementation of  Algorithm  \ref{ARCalgonew}  as \textit{ARC-Dynamic}.
The numerical tests in this section compare  \textit{ARC-Dynamic}, with
the following  variants.
\begin{itemize}
 \item  {\textit{ARC-Full}:}  Algorithm  \ref{ARCalgonew} employing exact Hessians;
 \item {\textit{ARC-Sub}:} Algorithm  \ref{ARCalgonew} employing inexact Hessian $B_k$ and accuracy $C_k=\epsilon$, for all $k\ge 0$, 
 i.e.
\begin{equation}\label{roosta}
\|\nabla^2 f(x_k)-B_k\|\le \epsilon,\quad \forall k\ge 0,
\end{equation}
as suggested in \cite{Roosta, Roosta_2p};
\item   {\textit {ARC-KL}: }Algorithm  \ref{ARCalgonew}  employing inexact Hessian $B_k$ and accuracy $C_k=\chi \|s_{k-1}\|$, for all $k\ge 1$.
In other words, we use the accuracy requirement \eqref{AMLK} 
replacing, as suggested in \cite{kl}, the unavailable information  $\|s_k\|$ in the righthand side of  \eqref{AMLK}  with the norm of the step $s_{k-1}$, i.e.,
\begin{equation}\label{AMLK_impl}
\|\nabla^2 f(x_k)-B_k\|\le \chi \|s_{k-1}\|,\quad \forall k\ge 1.
\end{equation}

To make a fair comparison with \textit{ARC-Dynamic}, the sample size $|\mathcal{D}_0|$ is set equal to $10\%$   of the number of
samples, since the first step has not been computed yet. 
Moreover, $\chi$ is chosen so that the sample size $|\mathcal{D}_1|$ resulting from  \eqref{minbound} with $C_1=\chi\|s_0\|$  is  $10\%$ of the total number of
samples.
\end{itemize}
\vskip 5pt

The synthetic datasets are listed in Table \ref{TableSynth}. For each dataset,
the number $N$ of training samples, the feature dimension $d$ and the testing size $N_T$ are reported.  
We also display the  $2$-norm condition number $cond$  
of the Hessian matrix at the approximate first-order optimal point (computed with ARC method,  exact Hessian  and stopping tolerance $\epsilon=10^{-3}$) and the value  of the scalar $C$ selected. 

\begin{small}
\begin{table}[h] 
\begin{center}
\begin{tabular}{cccccc}
\toprule
Dataset & Training~$N$ & $d$  & Testing $N_T$ & $cond$  & $C$ \\  \hline
\midrule
Synthetic1 &   9000 & 100 & 1000 &  $2.5\cdot10^4$ &$1.0101$ \\
Synthetic2 &   9000 & 100 & 1000 &  $1.4\cdot10^5$ &$1.0343$ \\
Synthetic3 &   9000 & 100 & 1000 &  $4.2\cdot10^7$ &$1.0406$ \\
Synthetic4 &   90000 & 100 & 10000 &  $4.1\cdot10^4$ &$0.2982$\\
Synthetic6 &   90000 & 100 & 10000 &  $5.0\cdot10^6$ & $0.3184$\\
\bottomrule
\end{tabular}
\caption{Synthetic datasets. Number  of training samples ($N$),  feature dimension  ($d$),  number of testing samples ($N_T$),
$2$-norm condition number  of the Hessian matrix at computed solution ($cond$),
scalar $C$ used in forming Hessian estimates ($C$).}\label{TableSynth}
\end{center}
\end{table}
\end{small}


In Table  \ref{TableSynthALL} we  report the results on all the synthetic datasets obtained with 
\textit{ARC-Dynamic} and  values $C$ as in Table \ref{TableSynth}.  
Since the selection of the subsets $\mathcal{D}_k$ is made randomly (and uniformly) at each iteration,  statistics 
in the forthcoming tables are averaged over $20$ runs. We display:
the total number of iterations ({\rm n-iter}), the value of EGE at termination ({\rm EGE}), the worst ({\rm Save-W}), best ({\rm Save-B}) and mean ({\rm Save-M}) percentages of savings obtained by  \textit{ARC-Dynamic}  with respect to \textit{ARC-Sub} and  \textit{ARC-KL} in terms of EGE.
 To give more insights, in what follows we focus on Synthetic1 and Synthetic6 as they are  representative of what we have observed in our experimentation. 

In Tables \ref{TableSynth1} and \ref{TableSynth6} we report statistics for these problems  solved  with 
our  algorithm and  constant $C$ different from the value in Table \ref{TableSynth}; we refer to such runs as 
\textit{ARC-Dynamic($C$)}. We  duplicate the results given in  Table  \ref{TableSynthALL}
for sake of readibility.

In Figure \ref{Perf1KL} we additionally show  the decrease 
of the training loss and the testing loss versus  the number of EGE   and  in Figure \ref{GDsyntehticdata} we plot  the gradient norm versus EGE.
In all the Figures we consider    \textit{ARC-Dynamic}, \textit{ARC-Dynamic(C)},   \textit{ARC-Sub} and \textit{ARC-KL}. 
A representative run is considered for each method;
 in Figure  \ref{Perf1KL} we do not plot \textit{ARC-Dynamic$(1.00)$} as it overlaps with \textit{ARC-Dynamic}.

\begin{small}
\begin{table}[h]   
\begin{center}
\begin{tabular}{l|cc|ccc|cccc}
\toprule
Dataset & \multicolumn{2}{c}{\textit{ARC-Dynamic}} & \multicolumn{3}{c}{\textit{ARC-Sub}}&  \multicolumn{3}{c}{\textit{ARC-KL}} \\  
&  n-iter  & EGE & Save-W & Save-B & Save-M & Save-W & Save-B & Save-M\\ \hline
\midrule
Synthetic1 &   17.2  & 103.7 & 38\% & 53\% & 44\% & -5\% & 43\% & 20\% \\ 
Synthetic2  &   16.7  & 89.5 & 47\% & 63\% & 55\% & -33\% & 50\% & 18\% \\ 
Synthetic3 &   17.1  & 94.6 & 46\% & 61\% & 51\% & -11\% & 47\% & 20\% \\ 
Synthetic4 &   15.6  & 85.3 & 58\% & 62\% & 60\% & -21\% & 22\% & 5\% \\
Synthetic6  &   15.2  & 67.4 & 60\% & 66\% & 63\% & -5\% & 36\% & 16\% \\
\bottomrule
\end{tabular}
\caption{The columns are divided in three different groups. \textit{ARC-Dynamic}: average number of iterations ({\rm n-iter}) and EGE at termination. \textit{ARC-Sub}: worst ({\rm Save-W}), best ({\rm Save-B}) and mean ({\rm Save-M}) percentages of saving obtained by \textit{ARC-Dynamic} over \textit{ARC-Sub} on the synthetic  datasets. \textit{ARC-KL}: worst ({\rm Save-W}), best ({\rm Save-B}) and mean ({\rm Save-M}) percentages of saving obtained by \textit{ARC-Dynamic} over \textit{ARC-KL} on the synthetic  datasets.}   \label{TableSynthALL}
\end{center}
\end{table}
\end{small}

\begin{small}
\begin{table}[h]   
\begin{center}
\begin{tabular}{l|cc|cccc}
\toprule
Method & n-iter   & EGE & Save-W & Save-B & Save-M\\ \hline
\midrule
\textit{ARC-Dynamic} &   17.2  & 103.7 & 38\% & 53\% & 44\% \\ 
\textit{ARC-Dynamic$(0.50)$} &15.4   & 145.4 & 9\% & 29\% & 21\%\\
\textit{ARC-Dynamic$(0.75)$} &16.5   & 112.6 & 27\% & 46\% & 39\%\\
\textit{ARC-Dynamic$(1.00)$}  &16.8   & 104.0 & 36\% & 53\% & 43\%\\
\textit{ARC-Dynamic$(1.25)$} &18.7   & 115.1 & 26\% & 54\% & 37\%\\
\bottomrule
\end{tabular}
\caption{Synthetic1 dataset. Average number of iterations ({\rm n-iter}),   EGE,
and worst ({\rm Save-W}), best ({\rm Save-B}) and mean ({\rm Save-M}) percentages of saving obtained by \textit{ARC-Dynamic} over \textit{ARC-Sub}.}   \label{TableSynth1}
\end{center}
\end{table}
\end{small}

\begin{small}
\begin{table}[h]   
\begin{center}
\begin{tabular}{l|cc|cccc}
\toprule
Method & n-iter &  EGE & Save-W & Save-B & Save-M\\ \hline
\midrule
\textit{ARC-Dynamic} &   15.2  & 67.4 & 60\% & 66\% & 63\% \\
\textit{ARC-Dynamic$(0.25)$} &15.1   & 78.9 & 53\% & 59\% & 57\%\\ 
\textit{ARC-Dynamic$(0.50)$} &15.9   & 58.5 & 57\% & 70\% & 68\%\\
\textit{ARC-Dynamic$(0.75)$} &16.6    & 61.5 & 54\% & 73\% & 66\%\\
\textit{ARC-Dynamic$(1.00)$} &16.8   & 64.1 & 46\% & 74\% & 65\%\\
\bottomrule
\end{tabular}
\caption{Synthetic6 dataset. Average number of iterations ({\rm n-iter}),  EGE, and worst ({\rm Save-W}), best ({\rm Save-B}) and mean ({\rm Save-M}) percentages of saving obtained by \textit{ARC-Dynamic} over 
 \textit{ARC-Sub}.}   \label{TableSynth6}
\end{center}
\end{table}
\end{small}
\noindent

\begin{figure}[h]
\centering
\includegraphics[width=%
0.49\textwidth]{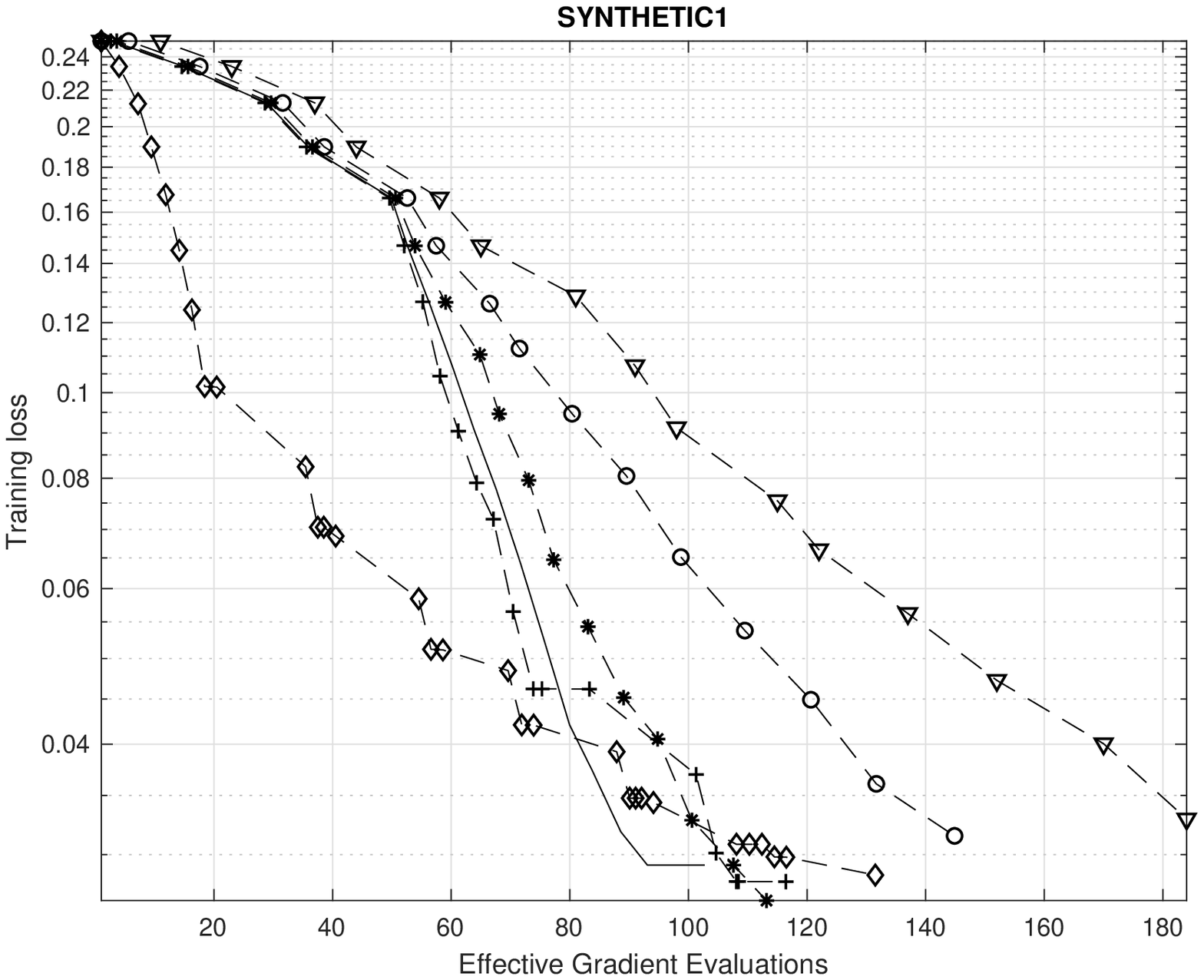}
\includegraphics[width=%
0.49\textwidth]{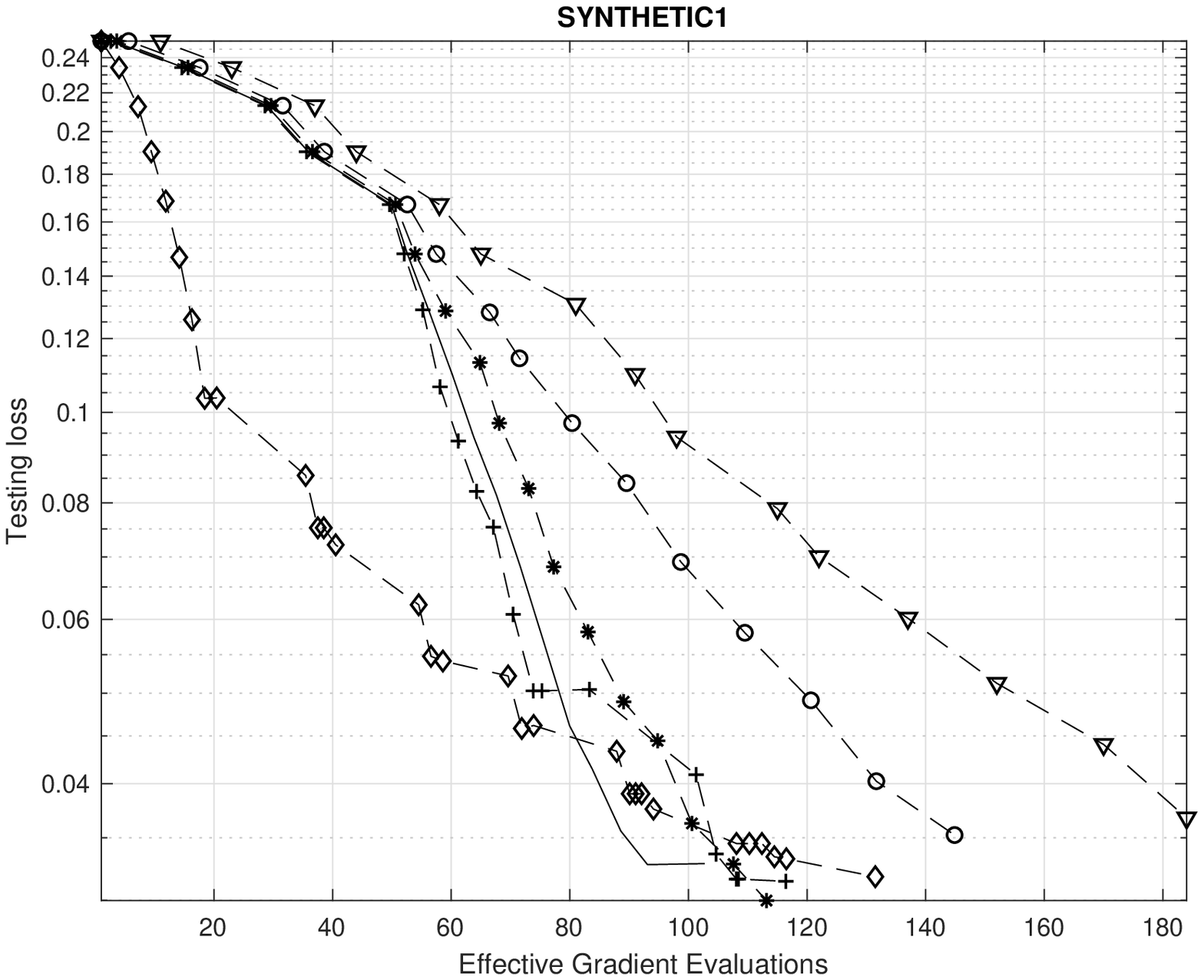}\
\includegraphics[width=%
0.49\textwidth]{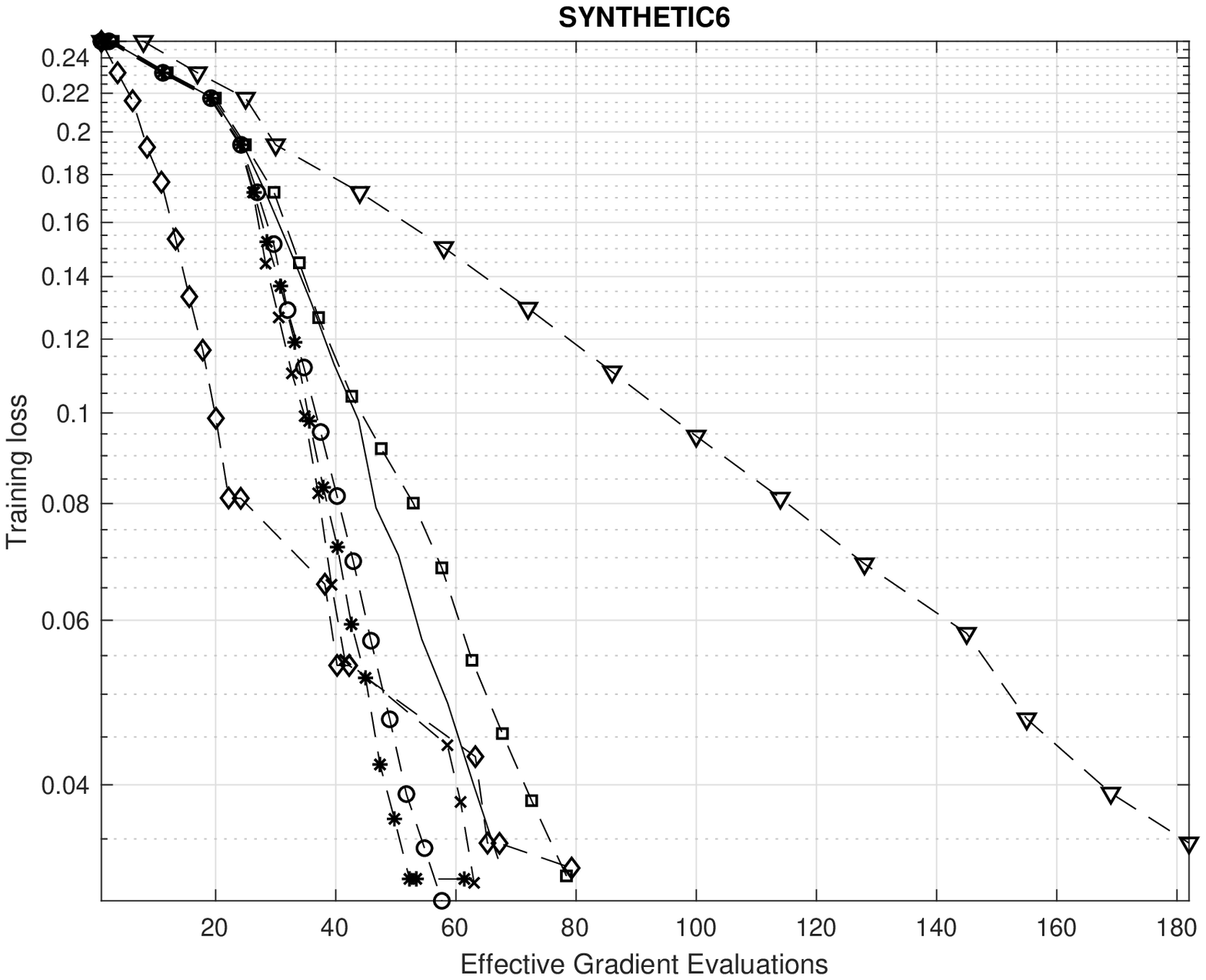}
\includegraphics[width=%
0.49\textwidth]{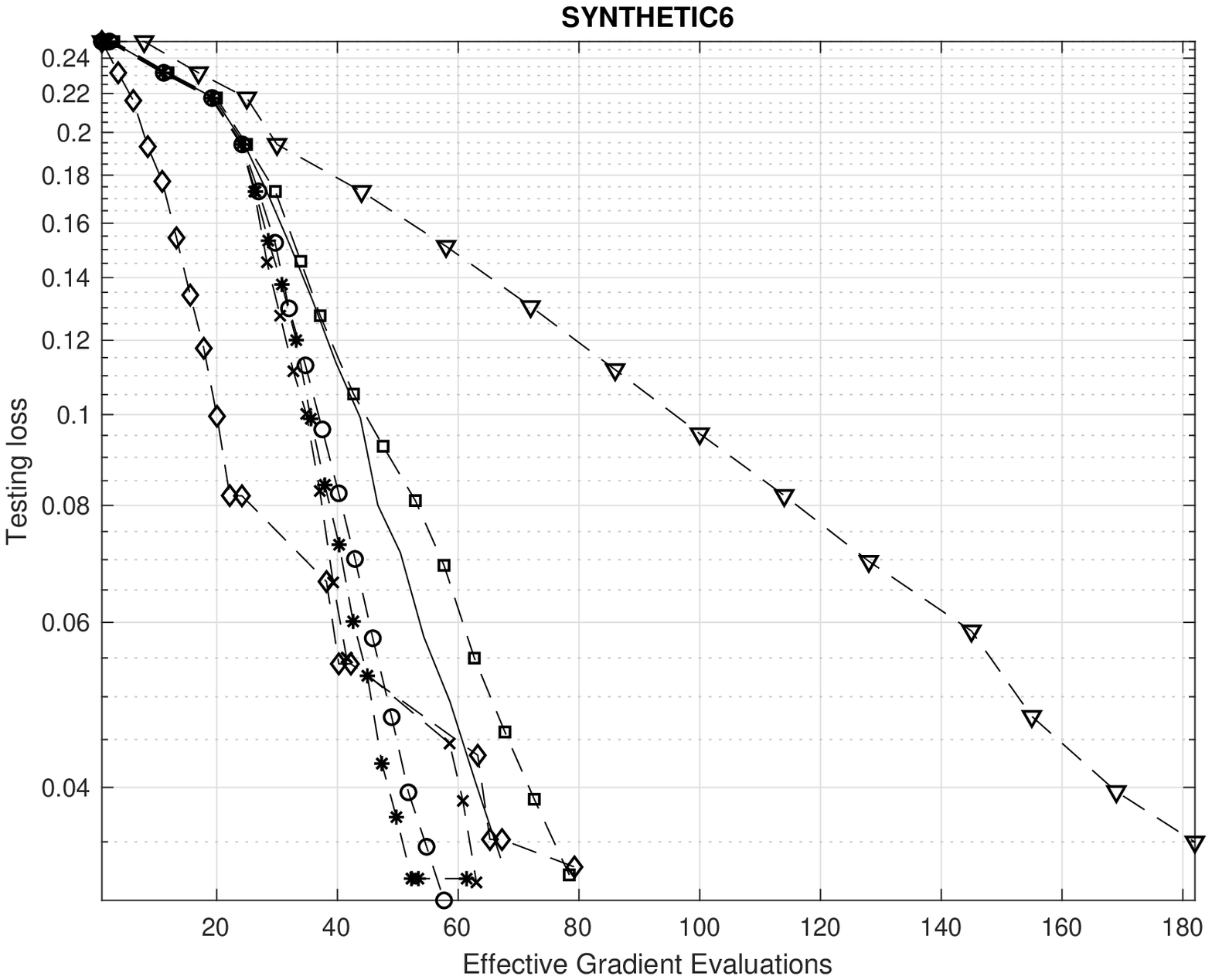}
\caption{Comparison of \textit{ARC-Dynamic} (continuous line), \textit{ARC-Dynamic C} with $C=0.25$ (dashed line with squares), $C=0.5$ (dashed line with circles), $C=0.75$ (dashed line with asterisks), $C=1$ (dashed line with crosses), $C=1.25$ (dashed line with plus symbols), ARC-KL (dashed line with diamonds) and \textit{ARC-Sub} (dashed line with triangles) against EGE. Each row corresponds to a different synthetic dataset.  training loss (left) and testing loss (right) against EGE,  logarithmic scale is on the $y$ axis.}
\label{Perf1KL}
\end{figure}

 Some comments are in order:
\begin{itemize}
\item  Condition \eqref{roosta} in \textit{ARC-Sub} yields a too high sample size at each iteration.  The adaptive strategies \textit{ARC-Dynamic} and  \textit{ARC-KL} outperform \textit{ARC-Sub} as in the latter algorithm  the
cost for computing  the Hessians is  not compensated by the gain in convergence rate.

\item Focusing on the two adaptive strategies \textit{ARC-Dynamic} and \textit{ARC-KL},
Table \ref{TableSynthALL} shows that on average  the former is less expensive than the latter. 
Figure \ref{Perf1KL} shows that \textit{ARC-KL} is  fast in the first stage of the convergence history,  becoming progressively slower as the norm of the step starts changing significantly from an iteration to the other (see Figure \ref{NormStepKL}).
In fact,  the  implementation of \textit{ARC-KL} relies on the assumption that $\|s_{k}\|$ is well approximated by $\|s_{k-1}\|$ and 
this is not always true. In particular,  Figure \ref{NormStepKL} shows that the norm of the step changes slowly initially while in the remaining 
iterations it oscillates and successive values differ by  some orders of magnitude.
This behaviour affects the euclidean norm of the gradient  as shown in Figure \ref{GDsyntehticdata}. We observe that such norm,
depicted against EGE, oscillates in \textit{ARC-KL}, while this is not the case 
in \textit{ARC-Dynamic} and \textit{ARC-Dynamic(C)}.

\item Focusing on our proposed adaptive strategy, Figure \ref{SampleSize} shows that  \textit{ARC-Dynamic} uses sets  $\mathcal{D}_k$ whose cardinality varies  
adaptively through iterations and it is considerably smaller than $N$ in most  iterations.    Moreover, the performance of \textit{ARC-Dynamic} appears to be quite insensitive  to  the choice of scalar $C$.
In fact, computational savings of \textit{ARC-Dynamic}  over  \textit{ARC-Sub} are achieved with various values of  
$C$, including those reported in Table \ref{TableSynth}.

\end{itemize}

\begin{figure}[h]
\centering
\includegraphics[width=%
0.49\textwidth]{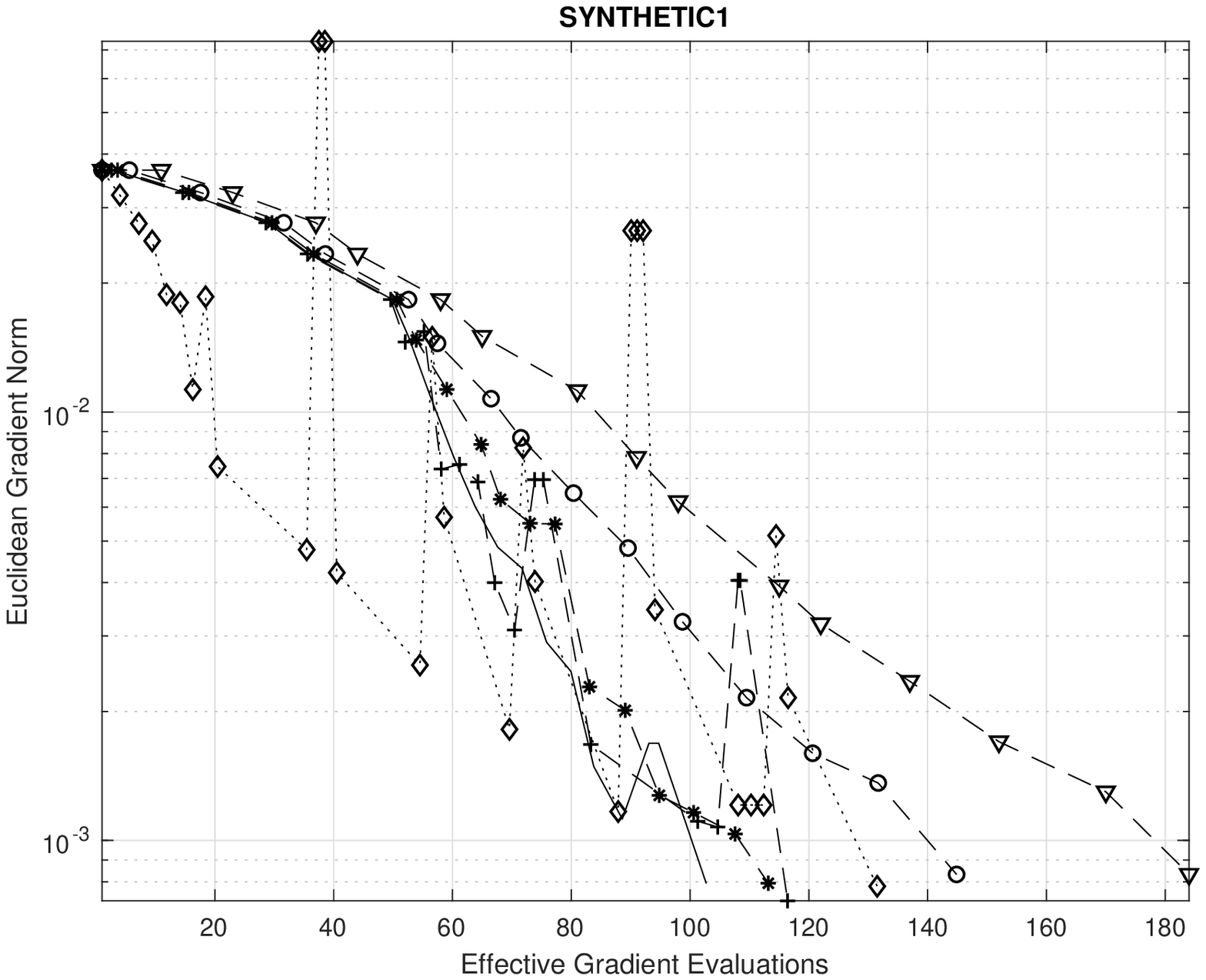}
\includegraphics[width=%
0.49\textwidth]{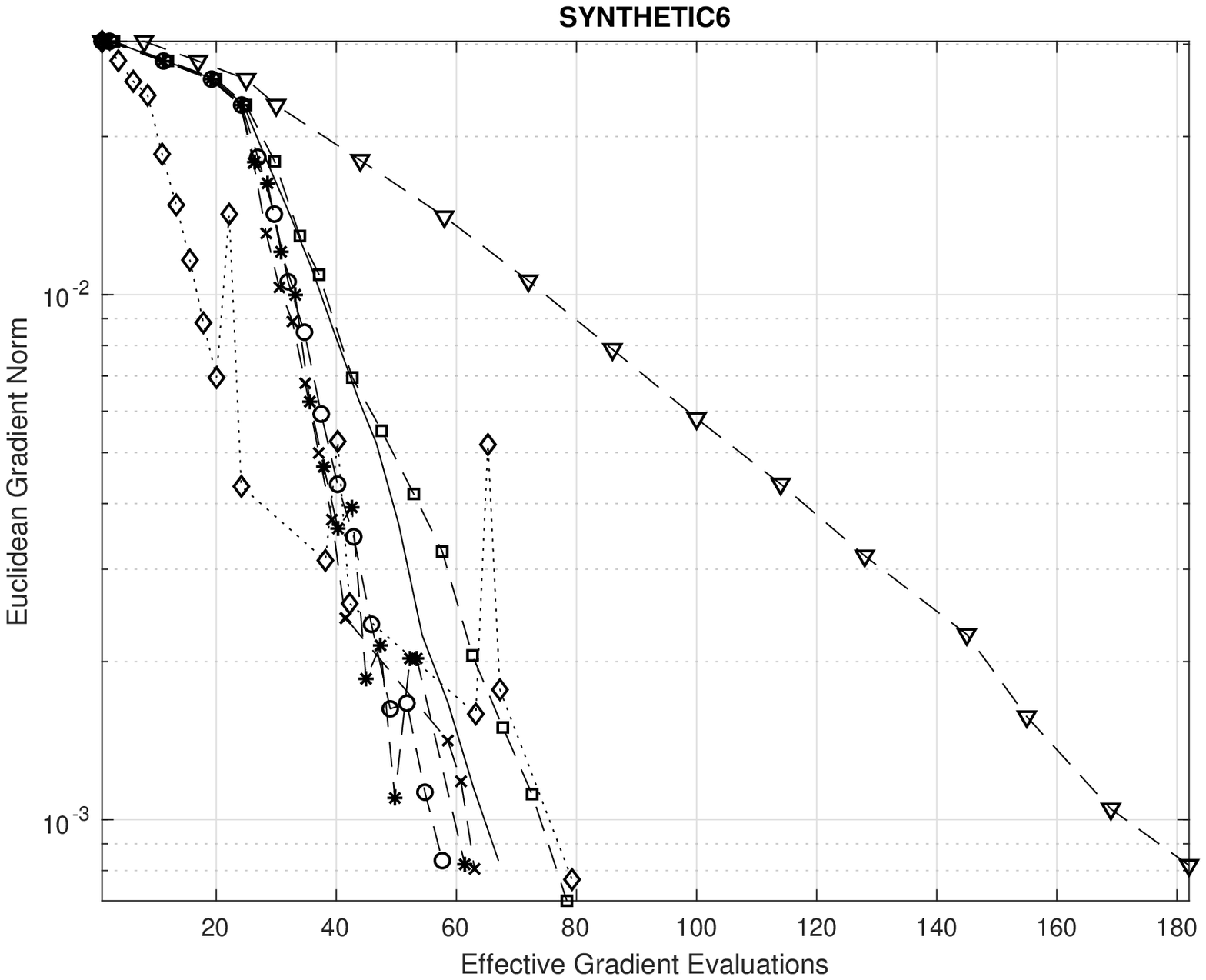}
\caption{Synthetic datasets, euclidean norm of the gradient against EGE (training set), logarithmic scale is on the $y$ axis. \textit{ARC-Dynamic} (continuous line), \textit{ARC-Dynamic C} with $C=0.25$ (dashed line with squares), $C=0.5$ (dashed line with circles), $C=0.75$ (dashed line with asterisks), $C=1$ (dashed line with crosses), $C=1.25$ (dashed line with plus symbols), \textit{ARC-KL} (dot line with diamonds) and  \textit{ARC-Sub} (dashed line with triangles).}
\label{GDsyntehticdata}
\end{figure}

\begin{figure}[h]
\centering
\includegraphics[width=%
0.49\textwidth]{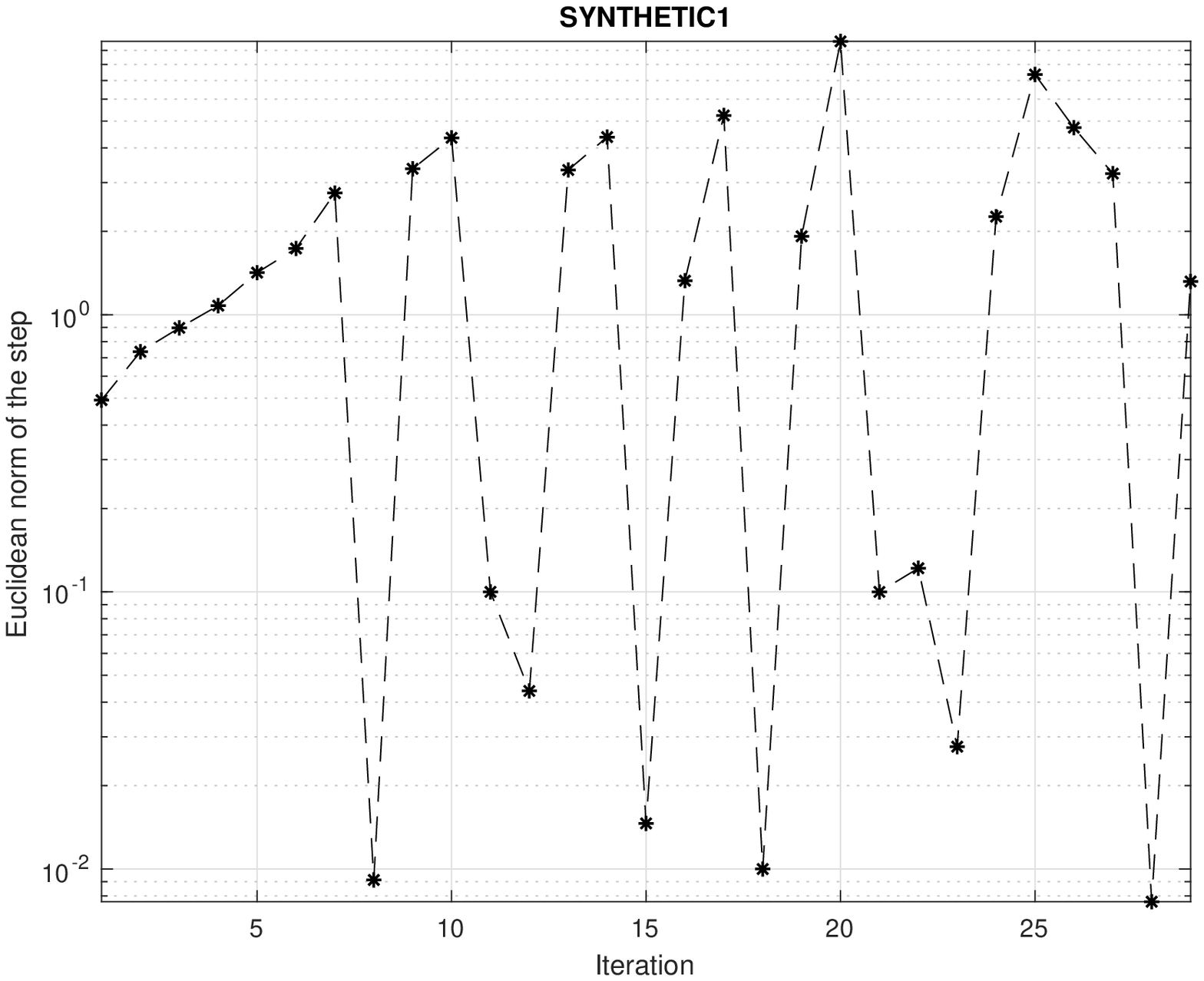}
\includegraphics[width=%
0.49\textwidth]{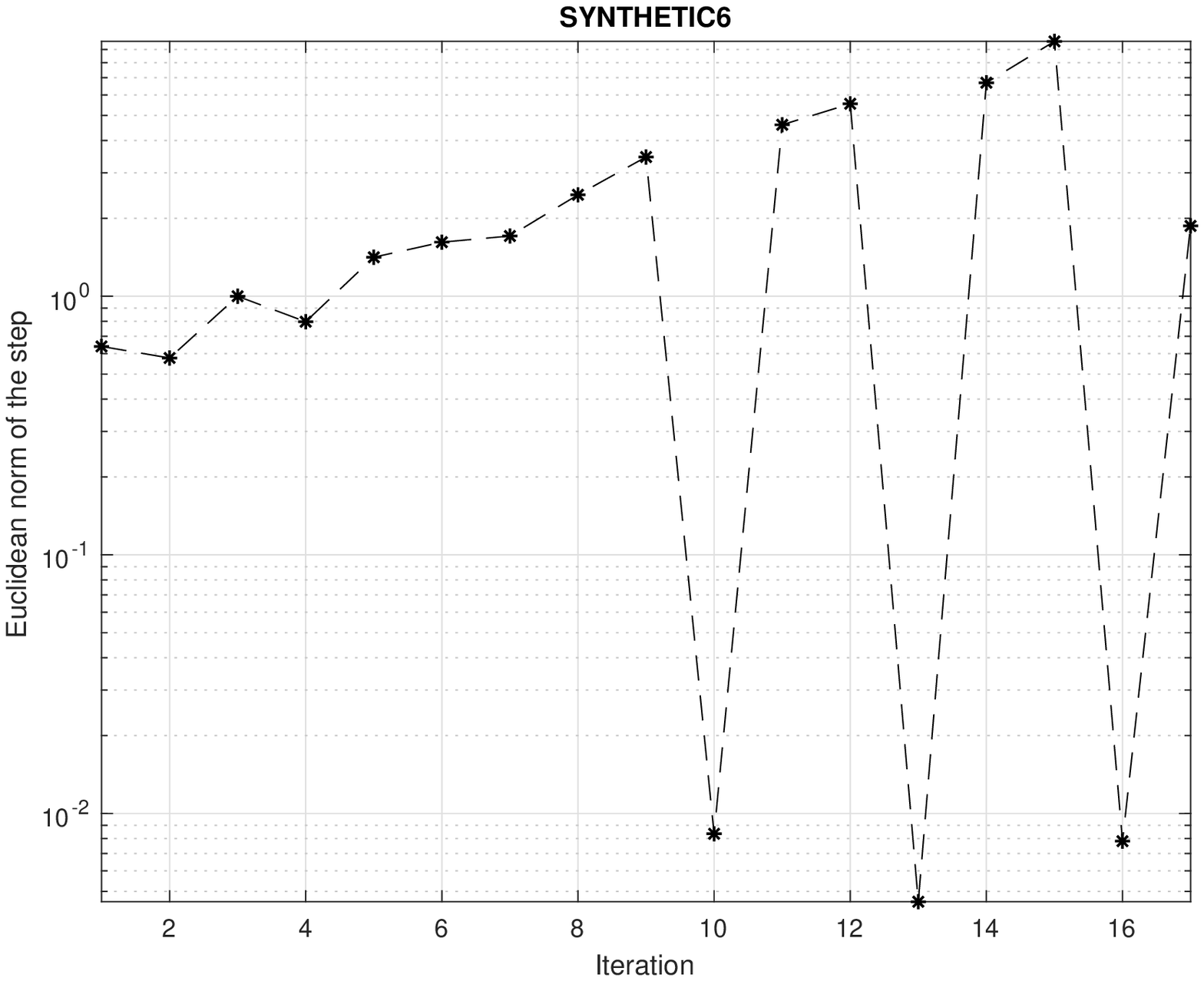}
\caption{Synthetic datasets, 2-norm of the step against iterations via \textit{ARC-KL}. Logarithmic scale on the $y$-axis.}
\label{NormStepKL}
\end{figure}

\begin{small}
\begin{table}[h]   
\begin{center}
\begin{tabular}{l|ccccc}
\toprule
Method & Synthetic1  & Synthetic2 & Synthetic3 & Synthetic4 & Synthetic6 \\\hline
\midrule
\textit{ARC-Dynamic}            & 98.00\%  & 96.80\% & 97.10\%  & 97.85\% & 97.98\%  \\
\textit{ARC-Dynamic(0.25)}   & ---            & ---          & ---            & 98.09\% & 98.08\%  \\ 
\textit{ARC-Dynamic(0.50)}   & 97.60\%  & 96.40\% & 96.90\%  & 98.19\% & 98.23\%  \\
\textit{ARC-Dynamic(0.75)}   & 98.10\%  & 96.60\% & 97.20\%  & 98.02\% & 98.11\%   \\
\textit{ARC-Dynamic(1.00)}   & 97.20\%  & 96.60\% & 96.10\%  & 98.15\% & 97.96\%  \\
\textit{ARC-Dynamic(1.25)}   & 98.00\%  & 96.60\% & 96.90\%  & ---           & ---   \\
\textit{ARC-Sub}                    & 97.50\%  & 96.60\% & 97.00\%  & 98.13\% & 97.87\%  \\
\textit{ARC-KL}                     & 97.80\%  & 96.60\% & 96.70\%  & 98.13\% & 97.98\%  \\
\bottomrule
\end{tabular}
\caption{Synthetic datasets. Binary classification rate on the testing set employed  by \textit{ARC-Dynamic}, \textit{ARC-Dynamic(C)}, $C\in\{0.25,0.5,0.75,1,1.25\}$, \textit{ARC-KL} and \textit{ARC-Sub}, mean values over $20$ runs.}   \label{BinAsyntheticdataset}
\end{center}
\end{table}
\end{small}

\noindent
\begin{figure}[h]
\centering
\includegraphics[width=%
0.49\textwidth]{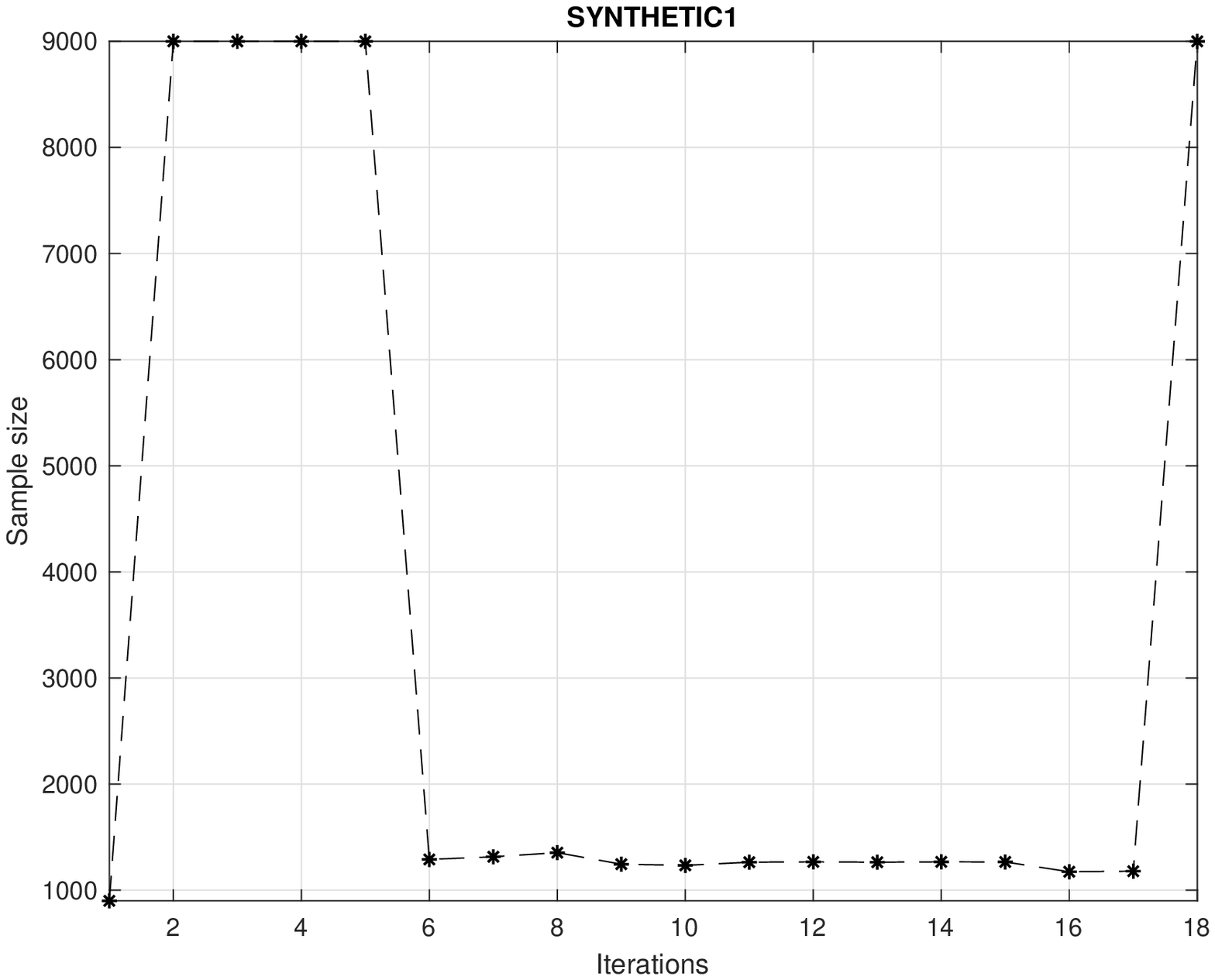}
\includegraphics[width=%
0.49\textwidth]{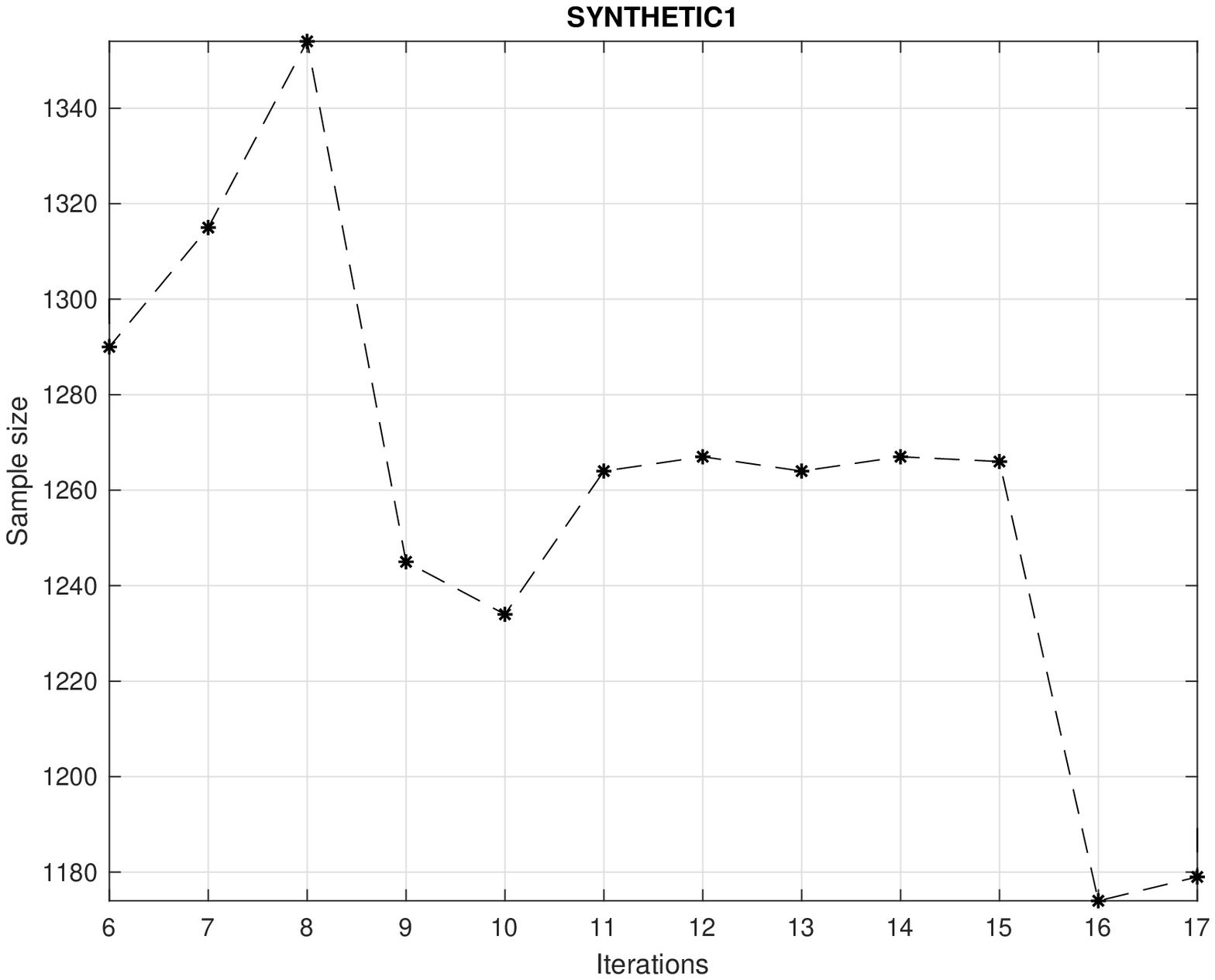}
\includegraphics[width=%
0.49\textwidth]{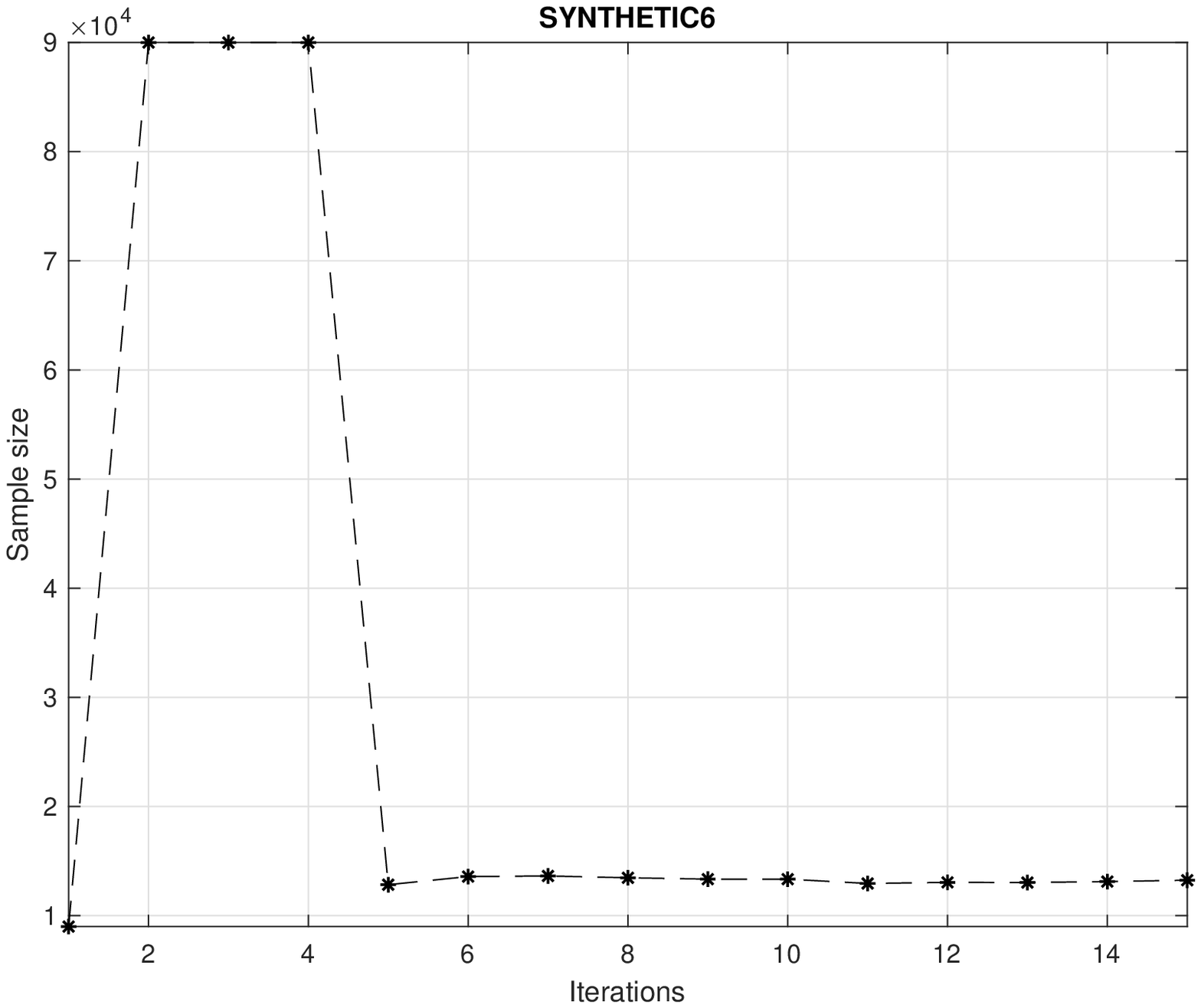}
\includegraphics[width=%
0.49\textwidth]{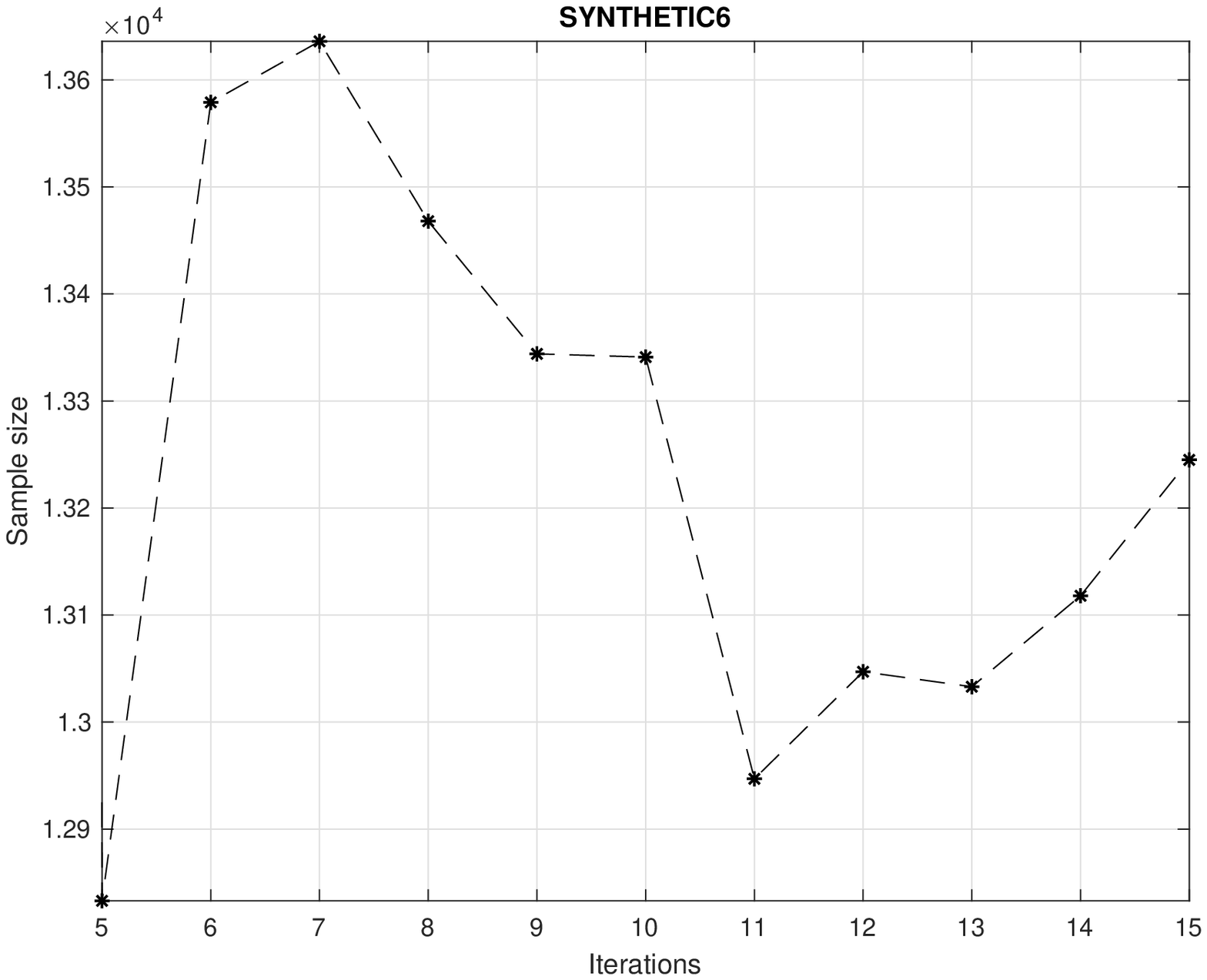}
\caption{Synthetic datasets. Sample size for Hessian approximations against iterations (left).
Portions of figures showing  low sample sizes (right).}
\label{SampleSize}
\end{figure}

The synthetic datasets used provide moderately ill-conditioned problems 
and motivate the use of second order methods.  Indeed, second order methods show their strength since 
all the tested procedures  manage to reduce the norm of the gradient and 
provide a small classification error. This is  shown in Table \ref{BinAsyntheticdataset} where
the average accuracy achieved by methods under comparison is reported. We outline that, the  difference between the percentages reported  in each column and  their  mean value  ranges from $0.20\%$ (best case) to $0,74\%$ (worst case), with an average of $0,38\%$. Thus, all the ARC variants reach a high accuracy in the testing phase and the preferable one is the variant requiring the lowest number of EGE at termination. 

As a final comment, our experiments show that, despite ill-conditioning, 
an accurate approximation of the Hessian is not required and  accuracy  dynamically chosen along iterations  works well in practice.
Adaptive thresholds for the Hessian approximations yield to procedures computationally more convenient than those using constant  and tiny thresholds and do not lack ability in solving the problems.

\subsection{{Real datasets}} \label{real}
In this section we present our second set of numerical results, performed on the machine learning datasets using 
subsampled ARC variants  with  deterministic suboptimal complexity of $O(\epsilon^{-2})$. 
More in depth, we compare our adaptive strategy with the version of ARC considered in \cite{Roosta_2p} where, at each iteration, 
the Hessian is approximated via subsampling on a set with prefixed and small cardinality. 

Our adaptive choice of   $| {\mathcal{D}_k}|$ is implemented   by introducing safeguards in    \eqref{minbound}.
Whenever $\|s_k\|<1$ we choose the cardinality of $\mathcal{D}_k$ 
according the following rule:
\begin{equation}\label{minbound2}
|\mathcal{D}_k|=\max\left \{ 0.05N,\min\left\{0.1N,\left\lceil\frac{4\rho}{C_k}
  \left(\frac{2\rho}{C_k}+\frac{1}{3}\right)
  \,\log\left(\frac{2n}{\overline{\delta}}\right)\right\rceil\right \}\right\},\quad k\ge 0,
\end{equation}
with $\rho>0$. Clearly,  $|\mathcal{D}_k|/N$ varies in the range $[0.05,0.1]$ for all $k\ge 0$, 
allowing us to compare our adaptive strategy with strategies employing fixed small  sample sizes.   The scalar 
$\rho$   is chosen so that $|\mathcal{D}_k|=0.1N$ when $C_k= \alpha (1-\theta) \epsilon^{2/3}$, i.e., 
the value of $C_k$ corresponding to $\|\nabla f(x_k)\|=\epsilon^{2/3}$ and $\|s_k\|<1$.
Whenever $\|s_k\|\ge 1$, the scalar $C$ used in \eqref{bound1}    is fixed so that $|\mathcal{D}_k|=0.05N$.

Guidelines for our rule  are:   sample size $0.05N$ is used when $\|s_k\|\ge 1$, 
larger sample size, up to $0.1N$, is  used  eventually.
Clearly, under this rule the Hessian sample size depends on the ratio $\rho/C_k$.

We compare  \textit{ARC-Dynamic} with the above choice of $|\mathcal{D}_k|$ against   its variant using  Hessian approximations obtained  
by subsampling on a small constant fraction of examples.
We will refer to the latter algorithm as    \textit{ARC-Fix(p)} where  $p\in(0,1)$ is the 
prefixed constant fraction of the $N$ examples used for building the Hessian approximations.

In  Table \ref{TableMLData} we  list the datasets used and 
for  sake of completeness, the  value of  the ratio  $\rho/C$ determining the Hessian sample size whenever $\|s_k\|\ge 1$ is
used. The MNIST dataset is here used for binary classification, 
labelling even digits with $1$ and odd digits with $0$. In the same table, in the column with header $\epsilon$ we report the used stopping tolerance. 
All test problems have been solved with $\epsilon=10^{-3}$
 except for Cina0 and HTRU2, where the tolerance has been increased to $10^{-2}$, since for lower values of $\epsilon$ we had no longer improvements on the decrease of the training and the testing loss, regardless of the method used. By contrast,  Mushroom was solved also using the tighter tolerance $\epsilon=10^{-5}$ as   below  threshold $\epsilon=10^{-3}$ further reduction in training and testing loss  was observed  and the percentage of failures in  classification on the testing set dropped from  1\%  to zero.
This can be observed in Table  \ref{BinAccrealdataset} where we report the average percentage of testing set data correctly classified. We also underline that,  
 the  gap between the percentages reported  in each column and  their  mean value 
 varies from $0\%$ (best case) to $0,89\%$ (worst case), for an average of $0,20\%$. Therefore, the different ARC methods considered achieve a high level of accuracy in the testing phase.
  \begin{small}
\begin{table}[h] 
\begin{center}
\begin{tabular}{l|ccc|c|c}
\toprule
Dataset & Training~$N$ & $d$  & Testing $N_T$ & $\epsilon$  & $\rho/C$\\  \hline
\midrule

Mushroom  \cite{UCI} &   6503 & 112 & 1621 & $10^{-3}$  & 2.3241   \\

 &   & & & $10^{-5}$  & 2.3241 \\
 
HTRU2  \cite{UCI} &   10000 & 8 & 7898 & $10^{-2}$  & 3.6942 \\

Cina0  \cite{CINA}&   10000 & 132 & 6033 & $10^{-2}$ & 2.8671 \\

Gisette  \cite{UCI}&   5000 & 5000 & 1000 & $10^{-3}$ &  1.6182\\

MNIST  \cite{mnist}&   60000 & 784 & 10000 & $10^{-3}$  & 6.3841\\

A9A  \cite{UCI}&   22793 & 123 & 9768 & $10^{-3}$  & 4.3922\\

Ijcnn1  \cite{libsvm}&   49990 & 22 & 91701 & $10^{-3}$ & 7.5283 \\

Reged0 \cite{libsvm} &   400 & 999 & 100 & $10^{-3}$ & 0.4443\\
\bottomrule
\end{tabular}
\caption{Real datasets. Size of the training set (Training $N$), problem dimension ($d$), size of the testing set (Testing $N_T$), tolerance $\epsilon$ for approximate optimality ($\epsilon$) and the ratio $\rho/C$ used for computing sample sizes.}\label{TableMLData}
\end{center}
\end{table}
\end{small}

 \begin{scriptsize}
\begin{table}[h]   
\begin{center}
\begin{tabular}{l|ccccccccc}
\toprule
Method &  \multicolumn{2}{c}{Mushroom} &\hspace*{-12pt} HTRU2 & Cina0 & Gisette & MNIST & A9A & Ijcnn1 & Reged0 \\ 
             & $\epsilon=10^{-3}$ & \hspace*{-5pt} $\epsilon=10^{-5}$ &  & & & & & & \\\hline
\midrule
\textit{ARC-Dynamic}  & 99.38\% &\hspace*{-20pt} 100\% & \hspace*{-20pt} 98.20\%  & 91.88\% & 97.40\%  & 89.92\%  & 84.81\%  & 91.76\%  & 96.00\%\\
\textit{ARC-Fix(0.01) } & 99.07\% & \hspace*{-20pt} 100\% & \hspace*{-20pt} 98.21\%  & 91.80\% & 97.60\%  & 89.84\%  & 84.83\%  & 91.95\%  & 96.00\%\\ 
\textit{ARC-Fix(0.05) } & 98.83\% & \hspace*{-20pt} 100\% & \hspace*{-20pt} 98.19\%  & 91.84\% & 97.50\%  & 89.83\%  & 84.76\%  & 91.75\%  & 96.00\%\\
\textit{ARC-Fix(0.1)}    & 99.32\% & \hspace*{-20pt} 100\% & \hspace*{-20pt} 98.20\%  & 91.88\% & 97.50\%  & 89.77\%  & 84.78\%  & 91.69\%  & 96.00\% \\
\textit{ARC-Fix(0.2)}    & 99.20\% &\hspace*{-20pt} 100\% & \hspace*{-20pt} 98.24\%  & 92.76\% & 97.30\%  & 89.82\%  & 84.83\%  & 91.70\%  & 96.00\% \\
\textit{ARC-Full}          & 98.77\% & \hspace*{-20pt} 100\% & \hspace*{-20pt}  98.27\%  & 93.10\% & 97.50\%  & 89.82\%  & 84.87\%  & 91.67\%  & 96.00\% \\
\bottomrule
\end{tabular}
\caption{Real datasets. Binary classification rate on the testing set employed  by \textit{ARC-Dynamic}, \textit{ARC-Fix(p)}, $p\in\{0.01,0.05,0.1,2\}$ and  \textit{ARC-Full}, mean values over $20$ runs.}  
 \label{BinAccrealdataset}
\end{center}
\end{table}
\end{scriptsize}

In Table  \ref{realdataset} we report, for each considered test problem and for each method under comparison, the  average 
number of EGE performed on 20 runs. 
We compare the performance of \textit{ARC-Dynamic} with that of \textit{ARC-Full} and \textit{ARC-Fix(p)}, $p\in\{0.01,0.05,0.1,0.2\}$.

\begin{small}
\begin{table}[h]   
\begin{center}
\begin{tabular}{l|ccccccccc}
\toprule
Method & \multicolumn{2}{c}{Mushroom}  & \hspace*{-12pt} HTRU2 & Cina0 & Gisette & MNIST & A9A & Ijcnn1 & Reged0 \\ 
             & $\epsilon=10^{-3}$ &  \hspace*{-5pt} $\epsilon=10^{-5}$ &  & & & & & & \\\hline
\midrule
\textit{ARC-Dynamic} & 29.8  & \hspace*{-20pt} 75.3 & \hspace*{-20pt}52.2  & 260.5 & 195.9  & 53.4  & 24.1 & 26.6  & 395.6\\
\textit{ARC-Fix(0.01) } & 41.5  & \hspace*{-20pt} 140.1 & \hspace*{-20pt}87.0 & 405.2 & 397.3 & 136.1 & 37.0 & 28.4 & 600.3\\ 
\textit{ARC-Fix(0.05) } & 35.5 & \hspace*{-20pt} 88.7& \hspace*{-20pt}86.2 & 335.6 & 221.0 & 101.5& 26.2 & 28.7  & 503.2\\
\textit{ARC-Fix(0.1)}  &39.6 & \hspace*{-20pt} 92.1  & \hspace*{-20pt}76.1  & 340.7 &231.0 &72.8 & 28.2 &31.3 &796.3 \\
\textit{ARC-Fix(0.2)}  & 38.1 & \hspace*{-20pt} 110.7 & \hspace*{-20pt}69.1& 453.4 & 268.8 & 73.5 & 34.5  & 36.1 & 1353.5 \\\
\textit{ARC-Full}  & 92.0 & \hspace*{-20pt} 264.0  & \hspace*{-20pt}158.0 & 2300.0 & 836.0 & 173.0 & 87.0 & 78.0 & 6932.0 \\
\bottomrule
\end{tabular}
\caption{Real datasets. Number of EGE employed  by \textit{ARC-Dynamic}, \textit{ARC-Fix(p)}, $p\in\{0.01,0.05,0.1,2\}$ and  \textit{ARC-Full}, mean values over $20$ runs.}   \label{realdataset}
\end{center}
\end{table}
\end{small}

Focusing on  the strategies employing a prefixed sample size, Table \ref{realdataset}  shows that
there is not a clear winner, as their performance depend on the specific dataset. However, all of them are clearly preferable to ARC with full Hessian, confirming that  uniformly sampling the Hessian on a low number of example is enough and there is no point to compute the full Hessian in these applications. On the other hand,  \textit{ARC-Dynamic} always terminates with the lowest number of EGE  and
gains over the most effective runs with   \textit{ARC-Fix(p)}    range  from 
11\% to 27\% in 7 out of 9 test problems and are larger  than 20\% in the solution of HTRU2, Cina0, MNIST, Reged0.   
This is confirmed by the 
 performance profile displayed in  Figure \ref{PerfProfile}. Denoting by $T$ the set of test problems in Table \ref{TableMLData}, by $S=\{$\textit{ARC-Dynamic}, $\{$\textit{ARC-Fix(p)}$\}$$_{p\in\{0.01,0.05,0.1,0.2\}}\}$ the set of the considered methods and by $E_{t,s}$ the number of EGE (at termination) to solve the problem $t\in T$ by the solver $s\in S$, the performance profile  \cite{dolmor2002}  for each $s\in S$ is defined as   the fraction  
$$
\rho_s(\tau)=\frac{1}{ |T|}\left |\Set{t\in T: r_{t,s}=\frac{E_{t,s}}{\min\{ E_{t,s}:s\in S \}}\le \tau }\right|,\qquad \tau\ge 1,
$$
of problems in $T$ solved by the method $s$ with a performance ratio $r_{t,s}$ within a fraction $\tau$ of the best solver. Comparing the values of $\rho_s(1)$, $s\in S$, it can be seen that \textit{ARC-Dynamic} outperforms the other solvers in the solution of all test problems. As already commented, the performances of the \textit{ARC-Fix(p)} methods are instead more controversial. 
More specifically, \textit{ARC-Fix(0.01)} and \textit{ARC-Fix(0.2)} seem to be overall less efficient, even if  \textit{ARC-Fix(0.01)}  is 
within a fraction $\tau = 1.08$ from the best solver on about $11\%$ of the problems. 
\textit{ARC-Fix(0.05)} solved all the problems within  $\tau=1.9$ while,  within such a value of $\tau$,
\textit{ARC-Fix(0.1)} and \textit{ARC-Fix(0.2)} solved $89\%$ of the problems  and  \textit{ARC-Fix(0.01)} solved $78\%$ of the problems. 
Moreover, \textit{ARC-Fix(0.2)} method requires a number of EGE which is within $\tau=3.4$ from the best one to solve all the problems. 
Finally, in  all runs  we observed  that the decreases of the training and   testing loss  with \textit{ARC-Dynamic}
is either comparable or faster than with  \textit{ARC-Fix(p)}. This features is displayed in Figures \ref{Perfmnist-gisette} where the
training and testing loss is plotted versus the number of EGE; representative runs reported concern datasets MNIST and Gisette.
\begin{figure}[h]
\centering
\includegraphics[width=%
0.7\textwidth]{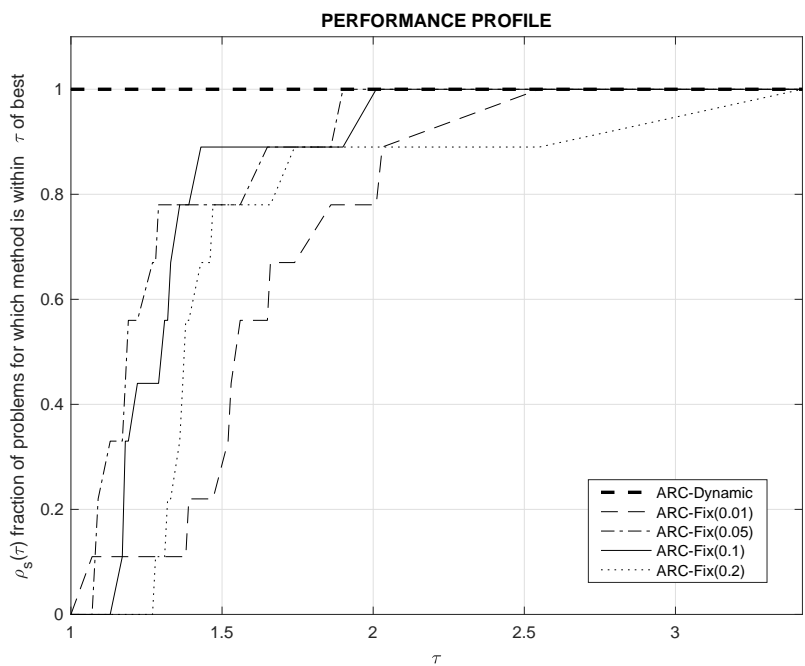}
\caption{Performance profile (EGE count) on $[1,~3.4]$ for real datasets.}
\label{PerfProfile}
\end{figure}
\begin{figure}[h]
\centering
\includegraphics[width=%
0.49\textwidth]{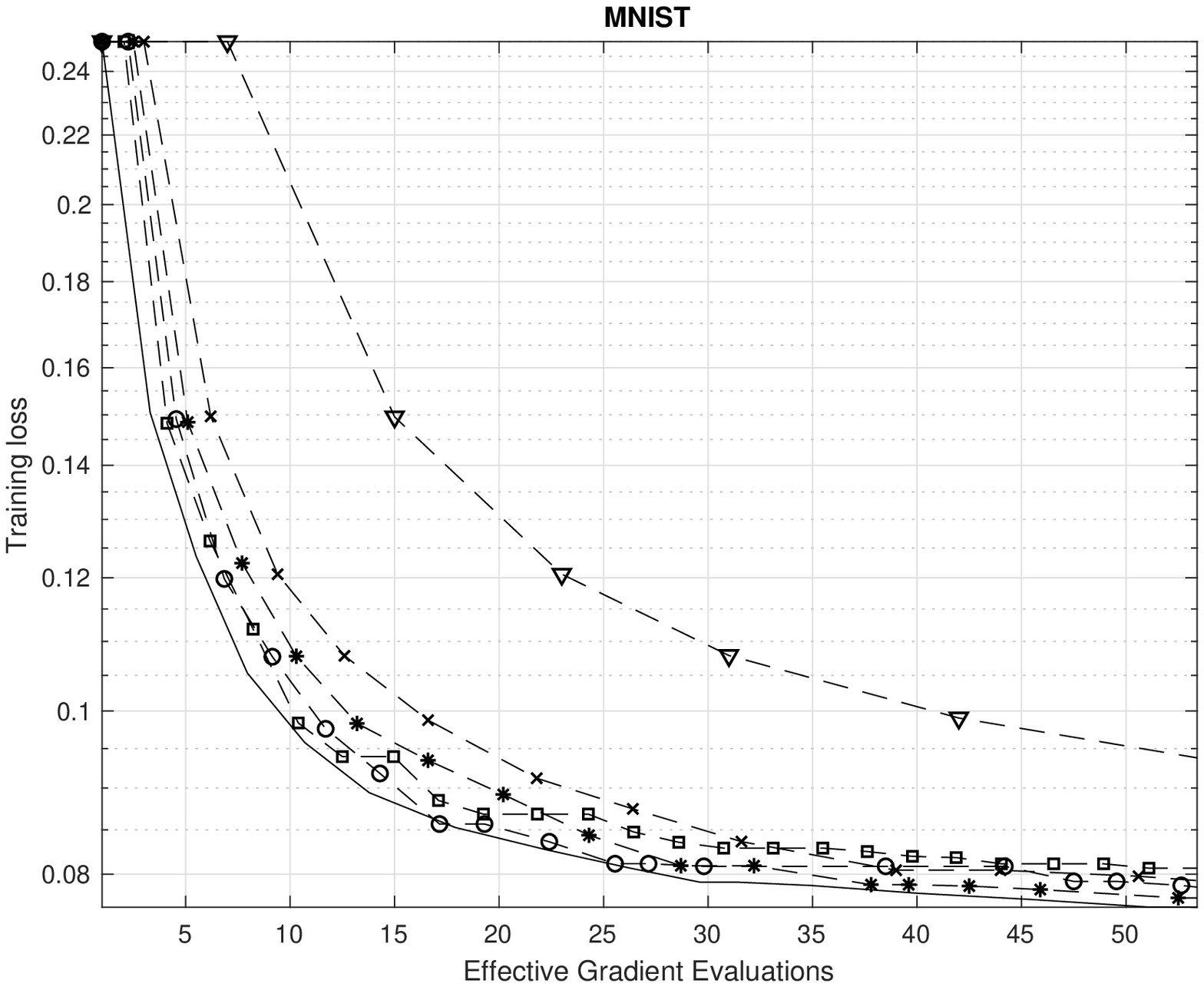}
\includegraphics[width=%
0.49\textwidth]{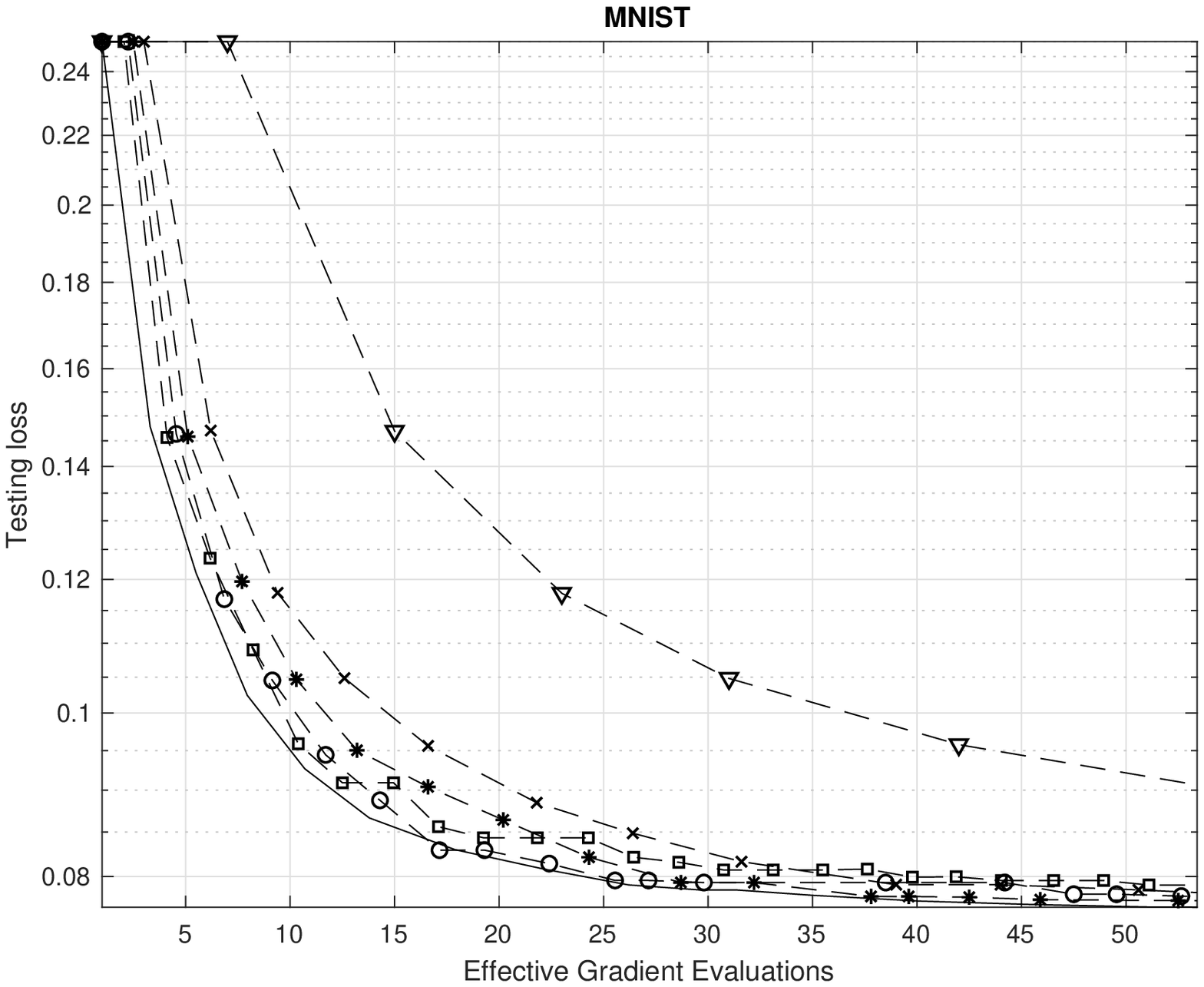}
\includegraphics[width=%
0.49\textwidth]{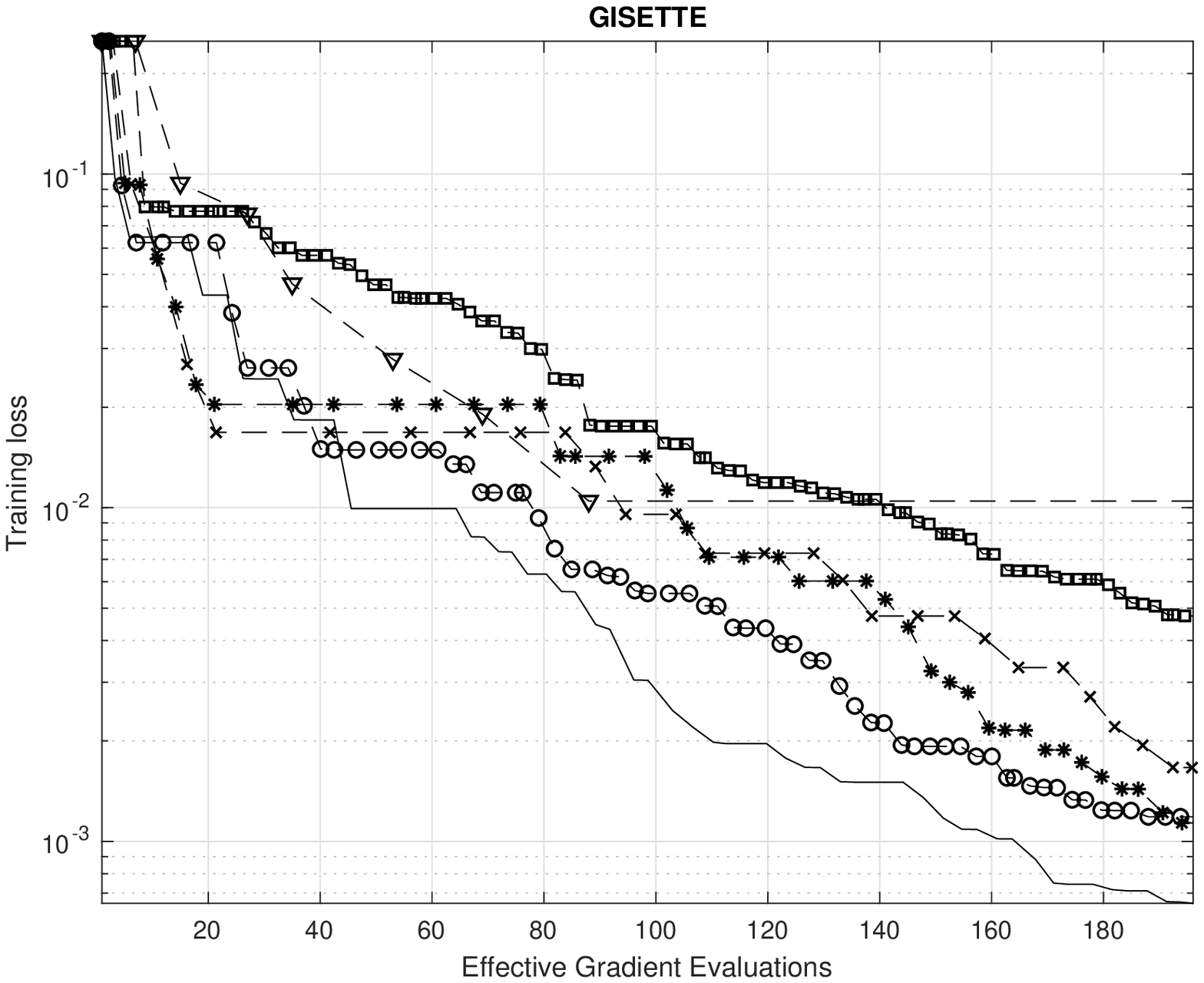}
\includegraphics[width=%
0.49\textwidth]{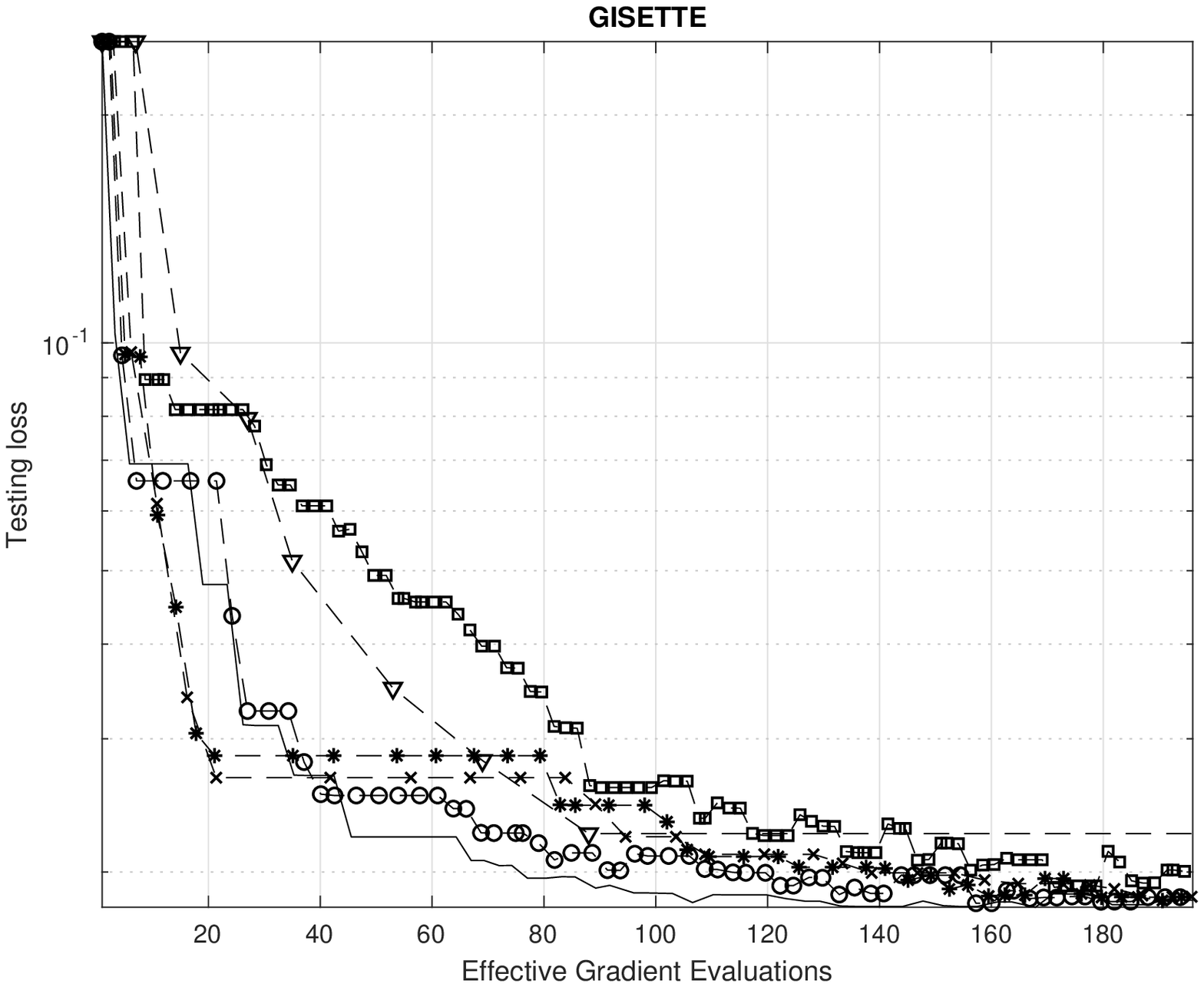}
\caption{MNIST dataset (top), Gisette dataset (bottom), 
training loss (left) and testing loss (right) against EGE,  logarithmic scale is on the $y$ axis.
\textit{ARC-Dynamic} (continuous line), \textit{ARC-Fix(p)} with $p=0.2$ (dashed line with crosses), $p=0.1$ (dashed line with asterisks), $p=0.05$ (dashed line with circles),  $p=0.01$ (dashed line with squares) and \textit{ARC-Full} (dashed line with triangles), }
\label{Perfmnist-gisette}
\end{figure}

%
%

%
%
\section{Conclusions and perspectives}
We proposed an ARC algorithm for solving nonconvex optimization problems based on a dynamic rule
for building inexact Hessian information. The new algorithm maintains the distinguishing features of ARC framework, i.e.,
the  optimal worst-case iteration bound for first- and second-order critical points.
Application to large-scale finite-sum minimization is sketched and analyzed.  

In case of  sums of strictly convex functions the adaptivity allows to improve complexity results in terms of 
component Hessian evaluations  over approaches that do not employ adaptive rules. 

We tested the new algorithm on a large number of problems  and compared its performance with the performance  of 
ARC variants with  optimal complexity  and  the performance  of ARC variants  employing a prefixed small Hessian sample size
and showing suboptimal complexity. The former 
comparison was carried out on synthetic moderately ill-conditioned datasets while 
the latter comparison was carried out on   machine learning datasets from the literature. 
Numerical results highlight that adaptiveness allows to 
reduce the overall computational effort and that the performance of the  proposed  method is quite  problem independent while  strategies taking a prefixed fraction of samples 
require a trial and error procedure to set the most efficient sample size.

 Convergence properties are analyzed both under deterministic and  probabilistic conditions, in the latter case properties of the deterministic algorithm 
are preserved in high probability.  However, this analysis does not give indication on the properties of the method when
the adaptive accuracy requirement is not satisfied. A stochastic analysis, in the spirit of \cite{CS}, would be of interest and it is the topic of future research. Moreover, 
we here assume that the objective function and the gradient are exact.  Extensions of this approach to the case where both function and gradient are evaluated with adaptive accuracy  is desirable as well as the employment of variance reduction techniques.

{
\section{Acknowledgements}
The authors wish to thank Raghu Bollapragada for gently providing the synthetic datasets and the referees for their insightful comments.
}

\end{document}